\documentclass[11pt,twoside,leqno]{amsart}
\usepackage[utf8]{inputenc}
\usepackage{amsmath, amsthm, amscd, amsfonts, amssymb, graphicx, color}
\usepackage[bookmarksnumbered, colorlinks, plainpages]{hyperref}
\hypersetup{colorlinks=true,linkcolor=red, anchorcolor=green, citecolor=cyan, urlcolor=red, filecolor=magenta, pdftoolbar=true}
\usepackage{tikz-cd}
\newtheorem{thm}{Theorem}[section]
\newtheorem{lemma}[thm]{Lemma}
\newtheorem{prop}[thm]{Proposition}
\newtheorem{cor}[thm]{Corollary}
\theoremstyle{definition}
\newtheorem{defn}[thm]{Definition}
\newtheorem{eg}[thm]{Example}

\theoremstyle{definition}
\newtheorem{rmk}[thm]{Remark}
\theoremstyle{definition}
\newtheorem{note}[thm]{Note}
\theoremstyle{definition}

\numberwithin{equation}{section}
\DeclareMathOperator{\Span}{span}

\DeclareMathOperator{\cpccloc}{\mathcal{CPCC}_{loc}\big(\mathcal{A}, C^{\ast}_{\mathcal{E}}(\mathcal{D})\big)}
  \DeclareMathOperator{\cloc}{\mathcal{C}_{loc}(E, C^{*}_{\mathcal{E, F}}(\mathcal{D, O}))}

\DeclareRobustCommand{\rrho}{{\mathpalette\irrho\relax}}
\newcommand{\irrho}[2]{\raisebox{\depth}{$#1\rho$}}
\DeclareRobustCommand{\rvarphi}{{\mathpalette\irvarphi\relax}}
\newcommand{\irvarphi}[2]{\raisebox{\depth}{$#1\varphi$}}
\title[Radon-Nikodym theorem]{Stinespring's Theorem for Unbounded Operator valued Local completely positive maps and Its Applications}
\author{B. V. Rajarama Bhat, Anindya Ghatak and P. Santhosh Kumar}
\address{Indian Statistical Institute Bangalore, Stat. Math. Unit, R.V. College Post, Bengaluru-560059, India.}
\email{bhat@isibang.ac.in}
\email{anindya.ghatak123@gmail.com}
\email{santhosh.uohmath@gmail.com}
\date{\today}
\begin{document}
\begin{abstract}
    Anar A. Dosiev in 2008 (See \cite{Dosiev}) obtained a   Stinespring's theorem for local completely positive maps (in short: local CP-maps) on locally $C^{\ast}$-algebras. In this article  a suitable notion of minimality for this construction has been identified so as to ensure uniqueness up to unitary equivalence for the associated representation.   Using this a Radon-Nikodym type theorem for local completely positive maps has been proved.

    Further,  a Stinespring's theorem for unbounded operator     valued local completely positive maps  on Hilbert modules over locally $C^{\ast}$-algebras (also called as local CP-inducing maps) has been presented. Following a construction of  M. Joi\c{t}a,  a Radon-Nikodym type theorem for local CP-inducing maps has been shown. In both cases the Radon-Nikodym derivative obtained is a positive contraction on some complex Hilbert space with an upward filtered family of reducing subspaces.
\end{abstract}
\maketitle
\section{Introduction}

Norm closed $*$-subalgebras of the algebra of all bounded operators are known as $C^*$-algebras and they have a very well developed theory. In an attempt to study algebras of unbounded operators, the notion of locally $C^{\ast}$-algebras was introduced by Atushi Inoue \cite{Inoue}.
It has been studied under various settings and under different names like multinormed $C^*$-algebra, $\sigma\text{-}C^*$-algebras  etc. W. Arveson used them to construct `tangent algebras' of $C^*$-algebras \cite{Arveson 88}. D. Voiculescu \cite{Voiculescu} introduced pro-$C^{\ast}$-algebras, as  essential objects to construct non-commutative analog of various classical Lie groups. Later N. C. Phillips studied them systematically as  projective limits of inverse families of  $C^*$-algebras in the context of $K$-theory of $C^*$-algebras \cite{Phillips-88}. This way the theory of locally $C^*$-algebras has shown its usefulness in many different contexts.

The notion of completely positive (CP) maps appeared naturally while studying maps between $C^*$-algebras and has its use in classification of $C^*$-algebras, quantum information theory, quantum probability theory and other fields.  The Stinespring's theorem \cite{Stinespring} is the basic structure theorem for completely positive maps.  There is plenty of literature on generalizing the notion of completely positive maps to maps between  locally $C^*$-algebras. See for example \cite{Dosiev12,Dosiev11,Dosiev,Joita-Main, Joita-corrigendum,Joita-tensorproduct,Moslehian et.al, Todorov et.al} and references therein. In  Theorem 2.1 of  \cite{Todorov et.al} we can find a generalization of Stinespring's theorem where the domain of the CP map is a dense subalgebra of a $C^*$-algebra and hence they are bounded elements. The unboundedness in that framework  is basically only  in the range and in the map itself.  A. Dosiev \cite{Dosiev} obtains a local CP version of Stinespring's theorem. We are considering his setting. Like the GNS representation theorem, one of the main features of Stinespring's theorem is a notion of minimality  which guarantees that minimal Stinespring's representation is unique  up to unitary equivalence. This is missing in Dosiev's work.
Here in  Theorem \ref{Theorem: UniquenessforDosiev}, we have proved such a result under suitable notion of minimality.
Looking at a pair of completely positive maps, one dominating the other, W. Arveson \cite{Arveson} obtained a non commutative analogue of Randon-Nikodym theorem. Here we generalize this result to local CP maps in Theorem \ref{Theorem: MainTheroem1}.

The study of Hilbert modules over locally $C^{\ast}$-algebras (in short: Hilbert locally $C^{\ast}$-modules) was
initiated by A. Mallios in \cite{Mallios} as a generalization of Hilbert $C^{\ast}$-modules. Later in \cite{Phillips-88}, N. C. Phillips
extended some basic properties of Hilbert $C^{\ast}$-modules to  Hilbert locally $C^{\ast}$-modules. Since then, there has been significant progress in the theory of Hilbert locally $C^{\ast}$-modules. In particular, Joita \cite{Joita-tensorproduct} defined tensor product of Hilbert locally $C^{\ast}$-modules and proved some results analogous to Hilbert $C^{\ast}$-modules. A. Gheondea \cite{Gheondea} constructed a notion of concrete Hilbert locally $C^{\ast}$-module by locally bounded operators
 (see \cite[Example 3.1 (2)]{Gheondea}), then showed that any Hilbert locally $C^{\ast}$-module is isomorphic to a concrete Hilbert locally $C^{\ast}$-module (see \cite[Theorem 3.2]{Gheondea}). We refer to the monograph of M. Joi\c{t}a \cite{Joita-monograph} for a survey of the theory of Hilbert locally $C^{\ast}$-modules.

We organize this article in six sections. In the second section, we recall some notations, basic definitions from theory of locally $C^{\ast}$-algebras and Hilbert modules over locally $C^{\ast}$-algebras. Also we state some existing results from the literature on local completely positive maps which are useful for later sections. In the third section, we introduce the notion of minimality for Stinespring's representation proved by Anar Dosiev \cite{Dosiev} and show the uniqueness of minimal Stinespring's representation. In the fourth section using the notion of minimality we prove a Radon-Nikodym type theorem for local CP-maps.

The fifth section contains a discussion of local CP-inducing maps on Hilbert locally $C^{\ast}$-modules and obtain Stinespring's representation for such maps. This is analogous to the result proved by Bhat, Ramesh and Sumesh in \cite{Bhat}. That result   was motivated by  Asadi \cite{Asadi}. Several authors have explored this idea and have extended the result in  different directions (See  \cite{Skeide1}, \cite{Skeide2},  \cite{Heo} and  \cite{Dey} ). In the final section, inspired by the work of  M. Joi\c{t}a ( \cite{Joita-Main}, \cite{Joita-corrigendum}) we define an equivalence relation on the class of all local CP-inducing maps and prove a Radon-Nikodym theorem for such maps.

\section{Notations and Preliminary results}

\subsection{Locally $C^{\ast}$-algebras}
\medskip
We recall some basic definitions in the context of locally $C^{\ast}$-algebras. We mostly follow the notation and terminology of \cite{Dosiev}.
 Let $\mathcal{A}$ be a unital $*$-algebra.
 \begin{enumerate}
     \item A seminorm $p$ on $\mathcal{A}$ is said to be \emph{sub-multiplicative}, if $p(1_{\mathcal{A}})=1$ and $p(ab)\leq p(a)p(b)$ for all $a, b\in \mathcal{A}$.
 \item A seminorm $p$ on $\mathcal{A}$ is said to be a \emph{$C^{\ast}$-seminorm}, if $p$ is sub-multiplicative and $p(a^{\ast}a) = p(a)^{2}$ for all $a\in \mathcal{A}$.
 \end{enumerate}
 Let $(\Lambda, \leq)$ be a directed poset and $\mathcal{P}:= \{p_{\alpha}: \alpha\in \Lambda\}$ be a family of $C^*$-seminorms defined on some $\ast$-algebra $\mathcal{A}$. Then $\mathcal{P}$ is called a \emph{upward filtered} family if
 \begin{equation*}
  p_{\alpha}(a)\leq p_{\beta}(a) \text{~~for all~~} a\in \mathcal{A}\; \text{and}\; \alpha\leq \beta.
 \end{equation*}
 \begin{defn}
 A \emph{locally $C^{\ast}$-algebra} $\mathcal{A}$ is a $\ast$-algebra which is complete with respect to the locally convex topology generated by a  upward filtered family $\{p_{\alpha}:\alpha\in \Lambda\}$ of $C^*$-seminorms defined on $\mathcal{A}$.
 \end{defn}

Throughout this article, we take $\mathcal{A}$ to be a locally $C^*$-algebra with a prescribed family  $\{p_{\alpha}: \alpha\in \Lambda\}$ of $C^*$-seminorms. We see that $\mathcal{A}$ is the \emph{projective limit} of an inverse family of $C^*$-algebras as follows: For each $\alpha\in \Lambda,$ let
 $\mathcal{I}_{\alpha}:= \{a\in \mathcal{A}: p_{\alpha}(a)=0\}.$ Clearly $\mathcal{I}_{\alpha}$ is a
 closed ideal in $\mathcal{A}$ and $\mathcal{A}_{\alpha}:= \mathcal{A}/\mathcal{I}_{\alpha}$ is a $C^*$-algebra with respect to the norm induced from $p_{\alpha}.$ For each $\alpha\leq \beta$, we define a map  $\pi_{\alpha\beta}: \mathcal{A}_{\beta}\to \mathcal{A}_{\alpha}$ by $\pi_{\alpha\beta}(a+ \mathcal{I}_{\beta})= a+ \mathcal{I}_{\alpha}$ for all $a\in \mathcal{A}$ and
 for each $\alpha$, we define a map $\pi_{\alpha}: \mathcal{A}\to \mathcal{A}_{\alpha}$ as the canonical quotient homomorphism. Then $\{\mathcal{A}_{\alpha}, \pi_{\alpha\beta}\}$ forms an inverse family of $C^*$-algebras because
 $\pi_{\alpha} =\pi_{\alpha\beta}\pi_{\beta}$ whenever
 $\alpha\leq \beta$.
 The projcetive limit
 $$\lim_{\xleftarrow[\alpha]{}} \mathcal{A}_{\alpha}:=\Big\{\{x_{\alpha}\}_{\alpha\in \Lambda}  \in \prod_{\alpha\in \Lambda}: \pi_{\alpha\beta}(x_{\beta})=
 x_{\alpha} \text{~~whenever~~} \alpha \leq \beta, \alpha, \beta\in \Lambda \Big\}$$ of the inverse family
 $\{\mathcal{A}_{\alpha}, \pi_{\alpha\beta}\}$
  is identified with $\mathcal{A},$ and the identification given by the map $a\to \{\pi_{\alpha}(a)\}_{\alpha\in \Lambda}.$
  For a systematic study of inverse limits of $C^*$-algebras, one can see \cite{Phillips-88}.

  An element $x\in\mathcal{A}$ is called
 \emph{self-adjoint} if $x^*=x$ and is called \emph{positive} if $x=y^*y$ for some $y\in \mathcal{A}.$ An important feature of local
 $C^*$-algebras is the notion of \emph{local self-adjoint} and \emph{local positive elements}. These elements occur naturally
 as $\mathcal{A}$ contains copies of each $\mathcal{A}_{\alpha}.$
\begin{defn}\cite{Dosiev} Let $a\in \mathcal{A.}$ Then $a$ is called
\begin{enumerate}
\item \emph{local self-adjoint} if $a=a^{*}+x,$ where $x\in \mathcal{A}$ such that $p_{\alpha}(x)=0$ for some $\alpha\in \Lambda$, and we call $a$ as \emph{$\alpha$-self-adjoint};
 \item \emph{local positive} if $a= b^*b+y$, where $b, y\in \mathcal{A}$ such that $p_{\alpha}(y)=0$ for some $\alpha\in \Lambda,$ we call $a$ as \emph{$\alpha$-positive}.
\end{enumerate}
\end{defn}
We write
$a\geq_{\alpha}0$ when $a$ is $\alpha$-positive. Moreover, we write $a\geq_{\alpha} b$ if $a-b\geq_{\alpha}0.$ It is easy to see that $a\in \mathcal{A}$ is \emph{local self-adjoint}
if $\pi_{\alpha}(a)$ is self-adjoint in $\mathcal{A}_{\alpha}$ for some $\alpha.$ Similarly, $a$ is locally positive if $\pi_{\beta}(a)$ is positive in
$\mathcal{A}_{\beta}$ for some $\beta.$ Also, notice that $a \geq_{\alpha} b$ if and only if $\pi_{\alpha}(a-b)\geq 0$ in $\mathcal{A}_{\alpha}.$
We write, $a=_{\alpha}0$ whenever $\pi_{\alpha}(a)=0$.
\subsection*{Quantized domain}
Throughout the paper, we write $\mathcal{H}$ as a complex Hilbert space, and $\mathcal{B(H)}$ as the set of all bounded operators on the Hilbert space $\mathcal{H}.$ In order to discuss representation of locally $C^*$-algebra, quantized domain is one of the important ingredient.

\begin{defn} Let $(\Omega, \leq )$ be a directed poset. A \emph{quantized domain} in a Hilbert space $\mathcal{H}$ is a triple $\{\mathcal{H}; \mathcal{E}; \mathcal{D}\}$, where
$\mathcal{E} = \{\mathcal{H}_{\ell}: \ell\in\Omega\}$ is an  upward filtered family of closed subspaces such that the union space $\mathcal{D}= \bigcup\limits_{\ell\in \Omega}
\mathcal{H}_{\ell}$ is dense in  $\mathcal{H}$.
\end{defn}

In short, we call $\mathcal{E}$ a \emph{quantized domain} in
$\mathcal{H}$ with its union space $\mathcal{D}$.  The quantized domain $\{\mathcal{H}; \mathcal{E}; \mathcal{D}\}$ is completely determined by the pair $(\mathcal{H}, \mathcal{E})$. However, for notational convenience we wish to retain $\mathcal{D}$ in the definition. Notice that this quantized family
$\mathcal{E}$ determines an upward filtered family $\{P_{\ell}: \ell\in \Lambda\}$ of projections in $\mathcal{B}(\mathcal{H}),$ where $P_{\ell}$ is a
projection from $\mathcal{H}$ onto $\mathcal{H}_{\ell}$.

Let $\mathcal{F}=\{ \mathcal{K}_{\ell}: \ell\in \Lambda\}$ be a quantized domain with
its union space $\mathcal{O}$ and $\mathcal{K}=\overline{\bigcup\limits_{\ell\in \Omega}\mathcal{K}_{\ell}}.$
Then $\mathcal{F}$ is  called a \emph{quantized subdomain} of $\mathcal{E}$, if $\mathcal{K}_{\ell}\subseteq \mathcal{H}_{\ell}$ for all $\ell\in \Omega.$
We express this collection of inclusions by simply writing $\mathcal{E}\subseteq \mathcal{F}$. Let $\mathcal{E}^{i}= \{\mathcal{H}_{\ell}^{(i)}: \ell\in\Omega\}$ be quantized
domain in a Hilbert space $\mathcal{H}^{i}$ with its union space $\mathcal{D}^{i}$ for $i=1,2$.
 Given a linear operator $V: \mathcal{D}^{1}\to \mathcal{H}^{2}$, we write $$V(\mathcal{E}_{1})\subseteq \mathcal{E}_{2} \text{~~if~~} V(\mathcal{H}_{\ell}^{(1)})\subseteq \mathcal{H}_{\ell}^{(2)} \text{~~for all~~} \ell\in \Omega.$$

The direct sum operation for quantized domains is defined as below. If $\mathcal{E}^{(i)}= \{\mathcal{H}_{\ell}^{(i)}: \ell\in\Omega\}$
be a quantized domain in the Hilbert space $\mathcal{H}^{(i)}$ with its union space $\mathcal{D}^{(i)}$ for $i=1,2\cdots n.$ Then $\bigoplus\limits_{i=1}^{n}\mathcal{E}^{(i)}=\Big\{\bigoplus\limits_{i=1}^{n} \mathcal{H}_{\ell}^{(i)}: \ell\in \Omega\Big\}$ is a quantized domain in the Hilbert space $\bigoplus\limits_{i=1}^{n}\mathcal{H}^{(i)}$ with the union space $\bigoplus\limits_{i=1}^{n}\mathcal{D}^{(i)}$. For simplicity, we denote the $n$-fold direct sum of copies of the quantized domain $\mathcal{E}$ by $\mathcal{E}^{n}.$


\subsection*{Unbounded operators on a Quantized domain}Before explaining, the algebra of unbounded operators on a quantized domain, let us quickly recall some basics of unbounded operator theory.

Let $\mathcal{H}$ and $\mathcal{K}$ be Hilbert spaces. A linear operator $T: \rm{dom}(T)\subseteq \mathcal{H}\to \mathcal{K}$ is said to be densely defined if $\rm{dom}(T)$ is a dense subspace of $\mathcal{H}$. The adjoint of $T$  is a linear map $T^{\bigstar} \colon \rm{dom}(T^{\bigstar})\subseteq \mathcal{K} \to \mathcal{H} $ with the domain given by
\begin{equation*}
    \rm{dom}(T^{\bigstar}):= \Big\{ \xi \in \mathcal{K}:\; \eta \mapsto \big\langle T\eta, \xi \big\rangle_{\mathcal{K}} \;\text{is continuous, for every}\; \eta \in \rm{dom}(T) \Big\},
\end{equation*}
and satisfying
\begin{equation*}
\langle T\eta, \xi \rangle_{\mathcal{K}}= \langle\eta, T^{\bigstar}\xi\rangle_{\mathcal{H}} ~~~\text{~for all~} \eta\in \rm{dom}(T), \text{~and~} \xi\in \rm{dom}(T^{\bigstar}).
\end{equation*}
The algebra of all \emph{noncommutative continuous functions} on a quantized
domain $\mathcal{E}$ is defined as
  \begin{align}\label{e1}
  \mathcal{C_{E}(D)}=\{T\in \mathcal{L(D)}: TP_{\ell} =P_{\ell} TP_{\ell}\in \mathcal{B(H)},~~ \text{for all}\; \ell\in \Omega \},\end{align}
where $\mathcal{L(D)}$ is the set of all linear operators on $\mathcal{D}$.
 If $T\in \mathcal{L(D)}$, then $T\in \mathcal{C_{E}(D)} $ if and only if
$T(\mathcal{H}_{\ell})\subseteq \mathcal{H}_{\ell}$ and $T\big|_{\mathcal{H}_{\ell}} \in \mathcal{B}(\mathcal{H}_{\ell})$, for all $\ell\in \Omega$. This implies that $\mathcal{C_{E}(D)}$ is an algebra.
 The $*$-algebra $\mathcal{C_{E}^*(D)}$ of $\mathcal{C_{E}(D)}$ is defined by
\begin{align}\label{e2}
 \mathcal{C_{E}^*(D)}=\{T\in\mathcal{C_{E}(D)}: P_{\ell}T\subseteq TP_{\ell},\; \text{for all}\;\ell\in \Omega\}.
\end{align}
 Notice that an element  $T\in \mathcal{C_{E}(D)}$ belongs to $\mathcal{C}_{\mathcal{E}}^*(\mathcal{D})$ if and only if
 $T(\mathcal{H}_{\ell}^{\perp}\cap \mathcal{D})\subseteq \mathcal{H}_{\ell}^{\perp}\cap \mathcal{D}$ for all $\ell\in \Omega.$
 It is shown in \cite[Proposition 3.1]{Dosiev} that if $T\in \mathcal{C_{E}^*(D)},$ then $\mathcal{D}\subseteq \rm{dom}(T^{\bigstar})$.
 Moreover, $T^{\bigstar}(\mathcal{H}_{\ell})\subseteq  \mathcal{H}_{\ell}$ for all $\ell \in \Omega.$ Now,
 let $$T^{*}= T^{\bigstar}|_{\mathcal{D}}.$$ Then it is easy to see that $T^*\in \mathcal{C_{E}^*(D)}.$ Thus $\mathcal{C_{E}^*(D)}$ is a unital $*$-algebra.
 Let us define seminorm $\Vert \cdot\Vert_{\ell}$ by
 \begin{align}\label{Equation: concrete Cstar-seminorm}
 \Vert T\Vert_{\ell}= \Vert T|_{\mathcal{H}_{\ell}}\Vert ~~~ \text{~for all~} T\in \mathcal{C_{E}^*(D)}.
 \end{align}
 Now, $\{\Vert\cdot\Vert_{\ell}: \ell\in \Omega \}$ is an upward filtered family of $C^*$-seminorms of $\mathcal{C_{E}^*(D)}$, thus $\mathcal{C_{E}^*(D)}$
 is a locally $C^*$-algebra determined by  the family $\{\Vert \cdot\Vert_{\ell }: \ell \in \Omega \}.$ If $T\in \mathcal{C_{E}^*(D)}$, then
 $T\geq_{\ell } 0$ if $T|_{\mathcal{H}_{\ell }}$ is positive operator in $\mathcal{B}(\mathcal{H}_{\ell }).$ Similarly, $T =_{\ell }0$ if $T|_{\mathcal{H}_{\ell }}=0$ in $\mathcal{B}(\mathcal{H}_{\ell }).$

 Let $\mathcal{A}$ and $\mathcal{B}$ be two locally $C^*$-algebras determined by two families of $C^*$-seminorms $\{p_{\alpha}: \alpha\in \Lambda\}$ and
 $\{q_{\ell}: \ell\in \Omega\}$ respectively. A linear map $\varphi: \mathcal{A}\to \mathcal{B}$ is called
 \begin{enumerate}
     \item \emph{local bounded} if for each $\ell\in \Omega$, there exist $\alpha \in \Lambda$ and $C_{\ell\alpha}>0$ such that $$q_{\ell}(\varphi(a))\leq C_{\ell\alpha} p_{\alpha}(a) ~~\text{~for all~} a\in \mathcal{A}.$$
     Note that $\varphi$ is continuous if and only if $\varphi$ is local bounded.
 \item \emph{local contractive}  if for each $\ell\in \Omega$, there exists $\alpha \in \Lambda$ such that $$q_{\ell}(\varphi(a))\leq p_{\alpha}(a) ~~\text{~for all~} a\in \mathcal{A};$$

 \item \emph{local positive} if for each $\ell\in \Omega$, there exists $\alpha\in \Lambda$ such that
 $$\varphi(a)\geq_{\ell} 0  \emph{~whenever~} a\geq_{\alpha}0.$$
  \end{enumerate}

It is very important to note that every locally $C^*$-algebra has a representation in $\mathcal{C_{E}^*(D)}$, for some $\{\mathcal{H}; \mathcal{E}; \mathcal{D}\}$. It can be treated as an
unbounded analog of Gelfand-Naimark theorem.
\begin{thm}\cite[Theorem 7.2]{Dosiev} Let $\mathcal{A}$ be a locally $C^*$-algebra, then there exists a quantized domain $\{\mathcal{H};\mathcal{E}; \mathcal{D}\}$ with a local isometrical $*$-homomorphism $\pi: \mathcal{A}\to \mathcal{C_{E}^*(D)}$.
\end{thm}
\begin{rmk} The above Gelfand-Naimark type theorem appeared in the work of Anar Dosiev in 2008 \cite{Dosiev},  in which he showed that each element of a
 locally $C^*$-algebra can be identified as an unbounded operator on a quantized domain.  Prior to that, A. Inoue showed the representation of  locally
 $C^{*}$-algebra in a little different approach \cite{Inoue}. Moreover, it was established that every locally $C^*$-algebra can be represented as a $*$-algebra of continuous linear operators on a \emph{locally Hilbert space}. Note that a locally Hilbert space is the inductive limit of directed family of Hilbert spaces, in which the topology obtained from inductive limit is a Hausdorff topology \cite{Gheondea}.
 \end{rmk}
 \begin{rmk} Representations of locally $C^*$-algebras have been studied either looking at quantized domains or considering locally Hilbert spaces. We  follow Dosiev's approach of studying representations on quantized domains.
 \end{rmk}

 Effros \cite{Eff-Web-96} initiated a study of  unbounded analog of operator space called  \emph{local operator space}, which is the locally convex version of operator space. It is a projective limit of inverse system of operator spaces. He introduced a class of morphisms on local operator spaces known as \emph{local completely bounded maps} (in short: we call it local CB-maps). On the other hand  Dosiev \cite{Dosiev}  extended the ideas of local operator space into local operator systems by a suitable notion of elements like local Hermitian, local positive elements etc. Moreover he introduced the class of morphisms on local operator system known as \emph{local completely positive maps} (in short: local CP-maps).


\subsection*{Local CB-maps and local CP-maps}
The definition local CP-map and local CB-map is defined between two local operator spaces and two local operator systems. This article is limited to study the structure theory of local CP-maps between two locally $C^*$-algebras and its applications.

Let $\mathcal{E}=\{\mathcal{H}_{\ell}: \ell\in \Omega\}$ be a \emph{quantized domain} in $\mathcal{H}$ with its union space $\mathcal{D}$. From earlier discussion, we know that $\mathcal{E}^{n}= \{\mathcal{H}^n_{\ell}: \ell\in \Omega\}$ is a \emph{quantized domain} in $\mathcal{H}^{n} := \mathcal{H}\oplus \cdots \oplus \mathcal{H}$ ($n$-copies of $\mathcal{H}$) with its union space
$\mathcal{D}^{n}$. Notice that $C^*_{\mathcal{E}^{n}}(\mathcal{D}^n)$ is a locally $C^*$-algebra. Moreover, $M_{n}(C^*_{\mathcal{E}}(\mathcal{D}))$ is isomorphic to
$C^*_{\mathcal{E}^{n}}(\mathcal{D}^n)$ via the following map
\begin{equation}
    [T_{i,j}]
    \left(\begin{bmatrix}
    \xi_{1}\\
    \vdots\\
    \xi_{n}
    \end{bmatrix}\right)
    = \begin{bmatrix}
    \sum\limits_{j=1}^{n}T_{1,j}\xi_{j}\\
    \vdots\\
    \sum\limits_{j=1}^{n}T_{n,j}\xi_{j}
    \end{bmatrix}.
\end{equation}
Let $[T_{i,j}]\in M_{n}(C^*_{\mathcal{E}}(\mathcal{D}))$. It is important to notice that
\begin{enumerate}
    \item $[T_{i,j}]\geq_{\ell} 0 \in M_{n}(C^*_{\mathcal{E}}(\mathcal{D}))$ if $[T_{i,j}]\big|_{\mathcal{H}_{\ell}^{n}}\geq 0$ in $\mathcal{B}(\mathcal{H}_{\ell}^{n});$ and
    \item $[T_{i,j}]=_{\ell}0$ if $[T_{i,j}]\big|_{\mathcal{H}_{\ell}^{n}}= 0$ in $\mathcal{B}(\mathcal{H}_{\ell}^{n}).$
\end{enumerate}

Let $\mathcal{A}$ and $\mathcal{B}$ be two locally $C^*$-algebras with associated families of $C^{\ast}$-seminorms $\{p_{\alpha}: \alpha\in \Lambda\}$ and $\{q_{\beta}: \beta\in \Omega\}$ respectively. Note that $M_{n}(\mathcal{A})$ and $ M_{n}(\mathcal{B})$ are locally $C^*$-algebras, where the associated families of $C^{\ast}$-seminorms are denoted by
$\{p^{n}_{\alpha}:\alpha\in \Lambda\}$ and $\{q^{n}_{\ell}: \ell\in \Omega\}$ respectively. Let $\varphi: \mathcal{A}\to \mathcal{B}$ be a linear map.  For each $n\in \mathbb{N}$, the \emph{$n$-amplification} of $\varphi$ is the map $\varphi^{(n)}: M_{n}(\mathcal{A}) \to M_{n}(\mathcal{B})$ defined by
$$
\varphi^{(n)}([a_{i,j}])= [\varphi(a_{i,j})] \text{~for all~} [a_{i,j}] \in M_{n}(\mathcal{A}).
$$
\begin{defn}[Local CB-map]Let $\varphi: \mathcal{A}\to \mathcal{B}$ be a linear map. Then $\varphi$ is called
\begin{enumerate}
\item \emph{local completely bounded} if for each $\ell \in \Omega$, there exist $\alpha \in \Lambda$ and $C_{\ell \alpha}>0$ such that $$q^{n}_{\ell}([\varphi(a_{i,j})])\leq C_{\ell \alpha} p_{\alpha}^{n}([a_{i,j}]),\; \text{for every}\; n\in \mathbb{N} ;$$
\item \emph{local completely contractive} if for each $\ell \in \Omega$, there exists $\alpha \in \Lambda$ such that
$$q^{n}_{\ell}[\varphi(a_{i,j})]\leq  p_{\alpha}^{n}([a_{i,j}]),\;\text{for every}\; n\in \mathbb{N}.$$
\end{enumerate}
\end{defn}
\begin{defn}[Local CP-map] Let $\varphi: \mathcal{A}\to \mathcal{B}$ be a linear map. Then $\varphi$ is called \emph{local completely positive}
if for each $\ell\in \Omega$, there exists $\alpha\in \Lambda$ such that
 $$\varphi^{(n)}([a_{ij}])\geq_{\ell} 0  \emph{~whenever~} [a_{ij}]\geq_{\alpha}0,$$
for every $n.$
\end{defn}

Now, we fix some notation for these classes of maps:
\begin{enumerate}
 \item $ \mathcal{CP}_{\text{loc}}\big(\mathcal{A}, C^{\ast}_{\mathcal{E}}(\mathcal{D})\big):$ the class of all local completely positive
 maps from $\mathcal{A}$ to $C^{\ast}_{\mathcal{E}}(\mathcal{D})$.
\item $ \mathcal{CC}_{\text{loc}}\big(\mathcal{A}, C^{\ast}_{\mathcal{E}}(\mathcal{D})\big):$ the class of all local completely contractive
maps from $\mathcal{A}$ to $C^{\ast}_{\mathcal{E}}(\mathcal{D})$.
\item $\cpccloc:$ the class of all local completely positive and local completely contractive maps from $\mathcal{A}$ to
$C^{\ast}_{\mathcal{E}}(\mathcal{D}).$
 \end{enumerate}

Our main interest is in    $\cpccloc$, the class of  local completely positive and locally completely contractive maps. Here is an example of such a map.

\begin{eg}
Let $\mathcal{H}$ be a complex, separable Hilbert space with a complete orthonormal basis: $\{e_{n}:\; n\in \mathbb{N}\}.$  Let $\{k_n\}_{n\in {\mathbb N}}$ be a monotonically
increasing sequence of natural numbers: $1\leq k_1<k_2<\cdots . $ For every $n \in \mathbb{N}$, we define $\mathcal{H}_{n} = \Span\{e_{1}, e_{2}, \hdots, e_{k_n}\}$, a closed subspace of $\mathcal{H}$. It is clear that $\{\mathcal{H}; \mathcal{E}; \mathcal{D}\}$ is a quantized domain in $\mathcal{H}$, where  $\mathcal{E} = \{\mathcal{H}_{n}:\; n\in \mathbb{N}\}$ is an upward filtered family and the union space $\mathcal{D} = \bigcup\limits_{n\in \mathbb{N}}\mathcal{H}_{n}$ is dense in $\mathcal{H}$.  The elements of $C^{\ast}_{\mathcal{E}}(\mathcal{D})$ are `block diagonal' operators which are possibly unbounded.  Let us consider an operator  $A \colon  \mathcal {H}\to \mathcal {H} $ which is positive and contractive.
Now we define the  \emph{Schur product} map $\psi _A$ (entry wise product map) initially on
${\mathcal B}({\mathcal H})$ by
$$ \psi _A([t_{ij}])= [a_{ij}t_{ij}]~~\mbox{for}~~ T=[t_{ij}]\in {\mathcal B}({\mathcal H})$$
where the operators are written in the given basis $\{e_n:n\in {\mathbb N}\}.$
Note that the Schur product $\psi _A(T)$ can be written using the isometry  $V \colon \mathcal{H} \to \mathcal{H} \otimes \mathcal{H},$ defined by $V(e_{n}) = e_{n} \otimes e_{n}, ~~\forall n \in \mathbb{N}$,  as
\begin{equation*}
    \psi _A( T)  = V^{\ast}(A\otimes T)V,
\end{equation*}
 Here the adjoint of $V$ is computed as,
\begin{equation*}
    V^{\ast}(y\otimes z) = \sum\limits_{i=1}^{\infty} \big\langle e_{i}\otimes e_{i},\; y\otimes z\big\rangle e_{i}, \; \text{for any}\; y,z \in \mathcal{H}.
\end{equation*}
From this it is easy to see that $\psi _A$ is a completely positive completely contractive map. Now we
extend the definition of Schur product to $C^{\ast}_{\mathcal{E}}(\mathcal{D}).$
It is clear that if $T$ leaves ${\mathcal H}_n$ invariant then so does $\psi _A(T)$.
 Define  $\varphi_{A} \colon C^{\ast}_{\mathcal{E}}(\mathcal{D}) \to C^{\ast}_{\mathcal{E}}(\mathcal{D})$ by
\begin{equation*}
    \varphi_{A}(T)|_{{\mathcal H}_n} = \psi _A (P_{{\mathcal H}_n}T|_{{\mathcal H}_n}), \; \text{for all}\; T \in C^{\ast}_{\mathcal{E}}(\mathcal{D}), n \in \mathbb{N}.
\end{equation*}
Equivalently $\varphi _{A}(T)=[a_{ij}t_{ij}]$ with respect to the given basis, thinking of $T$ as a block diagonal operator.
It is now easy to see that $\varphi _A$ is well-defined, local completely positive and local completely contractive.
\end{eg}

There is growing literature on the notion of completely positive maps on locally $C^{\ast}$-algebras (see for example \cite{Joita-Stinespring, Joita-StrictCP, Moslehian et.al, Maliev et.al} and references therein). Anar Dosiev \cite{Dosiev} studied local completely positive maps on locally $C^{\ast}$-algebras, using the notion of  local positive elements. He established a version of Stinespring's theorem for such maps. We recall the result here.

\begin{thm}\cite[Theorem 5.1]{Dosiev} \label{Theorem: DosievStinespring}
	Let $\varphi \in \cpccloc$. Then there exists a quantized domain $\big\{\mathcal{H}^{\varphi};\mathcal{E}^{\varphi}; \mathcal{D}^{\varphi} \big\}$, where  $\mathcal{E}^{\varphi}= \{\mathcal{H}^{\varphi}_{\ell}: \ell \in \Omega\}$ is an upward filtered family of closed subspaces of $\mathcal{H}^{\varphi}$,  a contraction $V_{\varphi}\colon \mathcal{H} \to \mathcal{H}^{\varphi}$ and a unital local contractive $\ast$-homomorphism $\pi_{\varphi} \colon \mathcal{A} \to C^{\ast}_{\mathcal{E}^{\varphi}}(\mathcal{D}^{\varphi})$ such that
	\begin{equation*}
	V_{\varphi}(\mathcal{E}) \subseteq \mathcal{E}^{\varphi} \;\; \text{and}\;\; \varphi(a) \subseteq V_{\varphi}^{\ast}\pi_{\varphi}(a)V_{\varphi},
	\end{equation*}
	for all $a \in \mathcal{A}$. Moreover, if $\varphi(1_{\mathcal{A}}) = I_{\mathcal{D}}$, then $V_{\varphi}$ is an isometry.
\end{thm}
\begin{rmk}
A local completely positive map need not be local completely bounded map. In order to establish a Stinespring type theorem in the category of local $C^{\ast}$-algebras, it is necessary to consider the class of local completely positive maps which are also local completely contractive (see \cite[Theorem 5.1]{Dosiev}).
\end{rmk}

\subsection{Hilbert modules over locally $C^{\ast}$-algebras}
We review some definitions and results from the literature on Hilbert modules over locally $C^{\ast}$-algebras which are needed
in this article.
\begin{defn}\cite{Gheondea}\label{Definition: innerproductAmodule}
Let $\mathcal{A}$ be a locally $C^{\ast}$-algebra and let $E$ be a complex vector space. A map $\langle \cdot , \cdot \rangle \colon E \times E \to \mathcal{A}$ is called an \emph{$\mathcal{A}$-valued inner product} if
\begin{enumerate}
    \item $\langle x, x\rangle \geq 0$, for all $x \in E$, and $\langle x, x\rangle = 0 $ if and only if $x=0$.
    \item $\langle x, \alpha y + z\rangle = \alpha \langle x, y\rangle + \langle x, z\rangle$, for all $\alpha \in \mathbb{C}$ and $ x,y,z \in E$.
    \item $\langle x,y\rangle^{\ast} = \langle y,x\rangle$, for all $x,y \in E$.
\end{enumerate}
Further, $E$ is said to be pre Hilbert $\mathcal{A}$-module if $E$ is a right $\mathcal{A}$-module compatible with the complex vector space structure:
\begin{equation*}
    \lambda (x a) = (\lambda x)a = x (\lambda a), \; \text{for all}\; x\in E, \lambda \in \mathbb{C}, a \in \mathcal{A}
\end{equation*}
and equipped with $\mathcal{A}$-valued inner product $\langle \cdot, \cdot \rangle \colon E\times E \to \mathcal{A} $ which is $\mathcal{A}$-linear in its second variable, that is
\begin{equation*}
    \langle x, y a \rangle = \langle x, y  \rangle a , \; \text{for all}\; x,y \in E \; \text{and}\; a \in \mathcal{A}.
\end{equation*}
\end{defn}
\begin{note} Let $\mathcal{A}$ be a locally $C^{\ast}$-algebra with respect to a prescribed family $\{p_{\alpha}: \alpha \in \Lambda\}$ of  $C^{\ast}$-siminorms and $E$ be a pre-Hilbert module over $\mathcal{A}$. Then by \cite[Proposition 1.2.2]{Joita-monograph}, the following \emph{Cauchy-Schwarz} inequality holds:
\begin{equation}\label{Equation: Cauchy-Schwarz}
    p_{\alpha}\big(\langle x, y\rangle\big)^{2} \leq p_{\alpha}\big(\langle x, x \rangle\big)\; p_{\alpha}\big(\langle y, y \rangle\big)
\end{equation}
for all $x,y\in E$ and $\alpha \in \Lambda$. Moreover, for each $\alpha \in \Lambda$, the map $\overline{p}_{\alpha} \colon E \to [0, \infty)$ defined by
\begin{equation*}
    \overline{p}_{\alpha}(x) = \sqrt{p_{\alpha}\big(\langle x,x \rangle\big)}, \; \text{for all}\; x \in E,
\end{equation*}
is a seminorm on $E$. Then $E$ is equipped with the topology generated by the family $\{\overline{p}_{\alpha}: \alpha \in \Lambda\}$ of seminorms, which is a Hausdorff locally convex topology on $E$. \end{note}
\begin{defn}
A pre-Hilbert $\mathcal{A}$-module $E$ is said to be a \emph{Hilbert $\mathcal{A}$-module} if $E$ is complete with respect to the topology generated by the family $\{\overline{p}_{\alpha}: \alpha \in \Lambda\}$ of seminorms on $E$. Moreover, $E$ is said to be \emph{full}, if
\begin{equation*}
    \overline{\Span}\big\{ \langle x,y \rangle: x,y\in E\big\}=\mathcal{A}.
\end{equation*}
\end{defn}
\subsection*{Hilbert locally $C^{*}$-module as a projective limit} We recall that  every Hilbert module over locally $C^{\ast}$-algebra can be seen as a projective limit of inverse system of Hlibert $C^{\ast}$-modules. Let $E$ be a Hilbert module over a locally $C^{\ast}$-algebra $\mathcal{A}$ and suppose that $\{p_{\alpha}: \alpha \in \Lambda\}$ is an associated family of continuous $C^{\ast}$-seminorms on $\mathcal{A}$. It is clear that, for each $\alpha \in \Lambda$, the set $${N}_{\alpha}:= \big\{x\in E:\; p_{\alpha}\big(\langle x,x \rangle) = 0\big\}$$ is a closed submodule of $E$ by the Cauchy-Schwarz inequality as shown in Equation (\ref{Equation: Cauchy-Schwarz}). Let us define $E_{\alpha}:= E/N_{\alpha}$, for each $\alpha \in \Lambda$. Then by \cite[Proposition 1.3.9]{Joita-monograph}, we see that $E_{\alpha}$ is a Hilbert module over $\mathcal{A}_{\alpha}$ with the norm given by
\begin{equation*}
    \|x+N_{\alpha}\|_{E_{\alpha}} = \inf\limits_{y\in N_{\alpha}} \overline{p}_{\alpha}\big( x+y\big), \; \text{for all} \; x \in E.
\end{equation*}
Since $N_{\beta} \subseteq N_{\alpha}$ whenever $\alpha \leq \beta$, there is a natural canonical projection $\sigma_{\alpha, \beta} \colon E_{\beta} \to E_{\alpha}$ defined by
\begin{equation*}
    \sigma_{\alpha, \beta} (x + N_{\beta}) = x+N_{\alpha}, \; \text{for all}\; x \in E.
\end{equation*}
Here $\sigma_{\alpha, \beta}$ is an $\mathcal{A}$-module map such that $\|\sigma_{\alpha, \beta}(x+N_{\beta})\|_{E_{\beta}} \leq \|x+N_{\alpha}\|_{E_{\alpha}}$ for all $x \in E$ and whenever $\alpha \leq \beta$. Suppose that $x,y \in E$ and  $a \in \mathcal{A}$, then for $\alpha \leq \beta$, we see that
\begin{align*}
    \sigma_{\alpha, \beta}\big((x+N_{\beta})(a+\mathcal{I}_{\beta})\big) = \sigma_{\alpha, \beta}\big( xa + N_{\beta}\big)
    &= xa + N_{\alpha} \\
    &= (x+N_{\alpha}) (a+\mathcal{I}_{\alpha})\\
    &= \sigma_{\alpha, \beta}\big((x+N_{\beta})\big)\pi_{\alpha, \beta}(a+\mathcal{I}_{\beta}),
\end{align*}
and
\begin{align*}
    \big\langle \sigma_{\alpha, \beta}(x+N_{\beta}),\; \sigma_{\alpha, \beta}(y+N_{\beta})\big\rangle
    &= \big\langle x+N_{\alpha},\;  y+N_{\alpha})\big\rangle \\
    &= \pi_{\alpha}(\langle x,y \rangle) \\
    &= \pi_{\alpha, \beta}\circ \pi_{\beta}(\langle x,y \rangle)\\
    &= \pi_{\alpha, \beta}\big( \langle x+N_{\beta},\; y+N_{\beta} \rangle\big).
\end{align*}
This implies that $\Big\{ E_{\alpha}; \sigma_{\alpha, \beta}; \mathcal{A}_{\alpha}; \pi_{\alpha, \beta} \Big\}_{\alpha \leq \beta, \alpha, \beta \in \Lambda}$ is an inverse system of Hilbert $C^{\ast}$-modules. Define
\begin{equation*}
    \lim_{\xleftarrow[\alpha]{}} E_{\alpha} := \Big\{\{x_{\alpha}\}_{\alpha \in \Lambda} \in \prod\limits_{\alpha \in \Lambda}E_{\alpha}:\; \sigma_{\alpha, \beta} (x_{\beta}) = x_{\alpha}\; \text{whenever}\; \alpha \leq \beta, \alpha, \beta \in \Lambda\Big\}.
\end{equation*}
Then $\displaystyle \lim_{\xleftarrow[\alpha]{}} E_{\alpha}$ is a Hilbert $\mathcal{A}$-module (see \cite[Proposition 1.2.21]{Joita-monograph}) and $\sigma_{\beta} \colon \displaystyle \lim_{\xleftarrow[\alpha]{}} E_{\alpha} \to E_{\beta}$ defined by
\begin{equation*}
    \sigma_{\beta}\big(\{x_{\alpha}\}_{\alpha \in \Lambda }\big) = x_{\beta},\; \text{for all}\; \{x_{\alpha}\}_{\alpha \in \Lambda }\in \displaystyle \lim_{\xleftarrow[\alpha]{}} E_{\alpha}
\end{equation*}
is $\mathcal{A}$-module map for each $\beta \in \Lambda$. The pair $\Big( \displaystyle \lim_{\xleftarrow[\alpha]{}} E_{\alpha}, \{\sigma_{\alpha}\}_{\alpha\in \Lambda}\Big)$ is known as the projective limit of inverse system  $\Big\{ E_{\alpha}; \sigma_{\alpha, \beta}; \mathcal{A}_{\alpha}; \pi_{\alpha, \beta} \Big\}_{\alpha \leq \beta, \alpha, \beta \in \Lambda}$ of Hilbert $C^{\ast}$-modules.

On the other hand, there is a canonical quotient map $\sigma^{E}_{\alpha}\colon E \to E_{\alpha}$, for each $\alpha \in \Lambda$. Then the pair $\Big(E, \{\sigma^{E}_{\alpha}\}_{\alpha \in \Lambda}\Big)$ is compatible to the inverse system $\Big\{ E_{\alpha}; \sigma_{\alpha, \beta}; \mathcal{A}_{\alpha}; \pi_{\alpha, \beta} \Big\}_{\alpha \leq \beta, \alpha, \beta \in \Lambda}$ of Hilbert $C^{\ast}$-modules in the sense that $\sigma_{\alpha, \beta} \circ \sigma_{\beta}^{E} = \sigma^{E}_{\alpha}$ whenever $\alpha \leq \beta$. Equivalently, for each $\alpha \leq \beta$, the following diagram commutes:
\begin{center}
    \begin{tikzcd}
    E \ar[dr, "\sigma_{\beta}^{E}"]\ar[rr, "\sigma_{\alpha}^{E}"]& & E_{\alpha}\\
    &E_{\beta} \ar[ur, "\sigma_{\alpha, \beta}"]&
    \end{tikzcd}
\end{center}
Hence by \cite[Proposition 1.3.10]{Joita-monograph}, the Hilbert $\mathcal{A}$-modules $E$ and $\displaystyle \lim_{\xleftarrow[\alpha]{}} E_{\alpha}$ are isomorphic through the map $\Gamma \colon E \to \displaystyle \lim_{\xleftarrow[\alpha]{}} E_{\alpha}$ defined by
\begin{equation*}
    \Gamma(x) = \big\{\sigma^{E}_{\alpha}(x)\big\}_{\alpha \in \Lambda}, \; \text{for all}\; x \in E.
\end{equation*}
Therefore, $E$ is a projective limit of the inverse system $\Big\{ E_{\alpha}; \sigma_{\alpha, \beta}; \mathcal{A}_{\alpha}; \pi_{\alpha, \beta} \Big\}_{\alpha \leq \beta, \alpha, \beta \in \Lambda}$ of Hilbert $C^{\ast}$-modules.

The constructions in the following example is  somewhat similar to  realizing a Hilbert $C^{\ast}$-module inside its linking algebra \cite{BGR}.
\begin{eg}\label{Example: Hilbertmodule}
Let $\big\{\mathcal{H}; \mathcal{E}; \mathcal{D}\big\}$ and $\big\{\mathcal{K}; \mathcal{F}; \mathcal{O}\big\}$ be two quantized domains, where
\begin{equation*}
    \mathcal{E}= \{\mathcal{H}_{\ell}:\; \ell \in \Omega\}, \; \mathcal{F}= \{\mathcal{K}_{\ell}:\; \ell \in \Omega\}
\end{equation*}
are upward filtered family of closed subspaces in $\mathcal{H}$ and $\mathcal{K}$ respectively. Motivated from the concrete locally $C^{\ast}$-algebra given in Equation (\ref{e2}), we define a class  of all \emph{noncommutative continuous functions} between quantized domains  $\big\{\mathcal{H}; \mathcal{E}; \mathcal{D}\big\}$ and $\big\{\mathcal{K}; \mathcal{F}; \mathcal{O}\big\}$ as,
\begin{align*}
    &C^{\ast}_{\mathcal{E}, \mathcal{F}}(\mathcal{D}, \mathcal{O})\\
    &:= \Big\{T\in \mathcal{L}(\mathcal{D}, \mathcal{O}):\; T(\mathcal{H}_{\ell}) \subseteq \mathcal{K}_{\ell}, T(\mathcal{H}_{\ell}^{\bot}\cap \mathcal{D}) \subseteq \mathcal{K}_{\ell}^{\bot}\cap \mathcal{O}\; \text{and}\; T\big|_{\mathcal{H}_{\ell}}\in \mathcal{B}(\mathcal{H}_{\ell}, \mathcal{K}_{\ell}),\; \text{for each}\; \ell \in \Omega \Big\}.
\end{align*}
Now we claim that $C^{\ast}_{\mathcal{E}, \mathcal{F}}(\mathcal{D}, \mathcal{O})$ is a Hilbert module over a locally $C^{*}$-algebra $C^{\ast}_{\mathcal{E}}(\mathcal{D})$. Note that $C^{\ast}_{\mathcal{E}, \mathcal{F}}(\mathcal{D}, \mathcal{O})$ is a complex vector space and has a natural right $C^{\ast}_{\mathcal{E}}(\mathcal{D})$-module  with the module map $C^{\ast}_{\mathcal{E}, \mathcal{F}}(\mathcal{D}, \mathcal{O}) \times C^{\ast}_{\mathcal{E}}(\mathcal{D}) \to C^{\ast}_{\mathcal{E}, \mathcal{F}}(\mathcal{D}, \mathcal{O})$ defined by
\begin{equation*}
    (T, S) \mapsto T S, \; \text{for all} \; T\in C^{\ast}_{\mathcal{E}, \mathcal{F}}(\mathcal{D}, \mathcal{O})\; \text{and} \; S\in C^{\ast}_{\mathcal{E}}(\mathcal{D}).
\end{equation*}
Define $\langle \cdot , \cdot \rangle \colon C^{\ast}_{\mathcal{E}, \mathcal{F}}(\mathcal{D}, \mathcal{O}) \times C^{\ast}_{\mathcal{E}, \mathcal{F}}(\mathcal{D}, \mathcal{O}) \to C^{\ast}_{\mathcal{E}}(\mathcal{D}) $ by
\begin{equation*}
    \langle T, S\rangle = T^{\ast}S,\; \text{for all}\; T,S \in C^{\ast}_{\mathcal{E}, \mathcal{F}}(\mathcal{D}, \mathcal{O}).
\end{equation*}
Then it satisfies properties $(1)$-$(3)$ of Definition \ref{Definition: innerproductAmodule}. Also, for every $T_{1}, T_{2}\in C^{\ast}_{\mathcal{E}, \mathcal{F}}(\mathcal{D}, \mathcal{O})$ and $S \in C^{\ast}_{\mathcal{E}}(\mathcal{D})$, we see that
\begin{equation*}
    \langle T_{1}, T_{2}S \rangle = T_{1}^{\ast}T_{2}S = \langle T_{1}, T_{2}\rangle S.
\end{equation*}
This implies that $C^{\ast}_{\mathcal{E}, \mathcal{F}}(\mathcal{D}, \mathcal{O})$ is a pre-Hilbert module over $C^{\ast}_{\mathcal{E}}(\mathcal{D})$. We know
that $\mathcal{B}(\mathcal{H}_{\ell}, \mathcal{K}_{\ell})$ is a right Hilbert module over the  $C^{\ast}$-algebra $\mathcal{B}(\mathcal{H}_{\ell})$ for every
$\ell \in \Omega$. By the nature of quantized domains and the defintion of $C^{\ast}_{\mathcal{E}, \mathcal{F}}(\mathcal{D}, \mathcal{O})$, there are canonical
projections defined by
\begin{equation*}
    \sigma_{\ell, \ell^{\prime}}\colon  \mathcal{B}(\mathcal{H}_{\ell^{\prime}}, \mathcal{K}_{\ell^{\prime}})\to \mathcal{B}(\mathcal{H}_{\ell}, \mathcal{K}_{\ell}) \; \text{and}\; \pi_{\ell, \ell^{\prime}}\colon \mathcal{B}(\mathcal{H}_{\ell ^{\prime}}) \to\mathcal{B}(\mathcal{H}_{\ell} ),
\end{equation*}
whenever $\ell \leq \ell^{\prime}$. Then $\Big\{\mathcal{B}(\mathcal{H}_{\ell}, \mathcal{K}_{\ell}); \sigma_{\ell, \ell^{\prime}}; \mathcal{B}(\mathcal{H}_{\ell}); \pi_{\ell, \ell^{\prime}}\Big\}_{\ell \leq \ell^{\prime}, \ell, \ell^{\prime}\in \Omega}$ is an inverse system of Hilbert $C^{\ast}$-modules and there is a canonical embedding of $C^{\ast}_{\mathcal{E}, \mathcal{F}}(\mathcal{D}, \mathcal{O})$ into the projective limit $\displaystyle \lim_{\xleftarrow[\ell]{}} \mathcal{B}(\mathcal{H}_{\ell}, \mathcal{K}_{\ell})$ which is closed.

Moreover, the locally convex topology on $C^{\ast}_{\mathcal{E}, \mathcal{F}}(\mathcal{D}, \mathcal{O})$ induced from  $\displaystyle \lim_{\xleftarrow[\ell]{}} \mathcal{B}(\mathcal{H}_{\ell}, \mathcal{K}_{\ell})$ is the topology generated by the family of seminorms $\{q_{\ell}: \ell \in \Omega\}$ defined by
\begin{equation*}
   q_{\ell}(T) = \sqrt{\|\langle T, T \rangle\|_{\ell}} = \sqrt{\big\|T^{\ast}T\big|_{\mathcal{H}_{\ell}}\big\|} = \big\|T\big|_{\mathcal{H}_{\ell}}\big\|_{\mathcal{B}(\mathcal{H}_{\ell}, \mathcal{K}_{\ell})}, \; \text{for every}\; T \in C^{\ast}_{\mathcal{E}, \mathcal{F}}(\mathcal{D}, \mathcal{O}),\; \ell \in \Omega.
\end{equation*}
Since $\displaystyle \lim_{\xleftarrow[\ell]{}} \mathcal{B}(\mathcal{H}_{\ell}, \mathcal{K}_{\ell})$ is complete (see \cite[Section 1.1]{Gheondea}), it follows that $C^{\ast}_{\mathcal{E}, \mathcal{F}}(\mathcal{D}, \mathcal{O})$ is complete. Hence $C^{\ast}_{\mathcal{E}, \mathcal{F}}(\mathcal{D}, \mathcal{O})$ is a Hilbert module over the locally $C^{\ast}$-algebra $C^{\ast}_{\mathcal{E}}(\mathcal{D})$.
\end{eg}
\begin{note} It is immediate to note that $T \in C^{\ast}_{\mathcal{E}, \mathcal{F}}(\mathcal{D}, \mathcal{O})$ if and only if $T^{\star} \in C^{\ast}_{\mathcal{F}, \mathcal{E}}(\mathcal{O}, \mathcal{D})$, where $T^{\ast}:= T^{\bigstar}\big|_{\mathcal{D}}$ similar to the case of $C^{\ast}_{\mathcal{E}}(\mathcal{D})$. In particular, if $\mathcal{E} = \mathcal{F}$, then $\mathcal{H}= \mathcal{K}$ and  $C^{\ast}_{\mathcal{E}, \mathcal{E}}(\mathcal{D}, \mathcal{D}) = C^{\ast}_{\mathcal{E}}(\mathcal{D})$ is a locally $C^{\ast}$-algebra as shown in \cite[Proposition 4.1]{Dosiev}.
\end{note}
\begin{defn}
Let $E$ be a Hilbert module over a locally $C^{\ast}$-algebra $\mathcal{A}$. Let $\{\mathcal{H};\mathcal{E};\mathcal{D}\}$ and $\{\mathcal{K};\mathcal{F};\mathcal{O}\}$ be quantized domains. A map $\Phi \colon E \to C^{\ast}_{\mathcal{E}, \mathcal{F}}(\mathcal{D}, \mathcal{O})$ is said to be
\begin{enumerate}
    \item a \emph{$\varphi$-map}  if   $\varphi \colon \mathcal{A} \to C^{\ast}_{\mathcal{E}}(\mathcal{D})$ is linear and
\begin{equation*}
    \big\langle \Phi(x), \Phi(y)\big\rangle = \varphi \big( \langle x, y\rangle\big),\; \text{for all}\; x,y \in {E}.
\end{equation*}

\item \emph{$\varphi $-morphism}, if $\Phi$ is a $\varphi $-map and $\varphi$ is a morphism.
\item     \emph{local CP-inducing map} if it is a $\varphi $-map for some $\varphi \in \cpccloc $.

\end{enumerate}
\end{defn}

\section{Minimality of Stinespring representation for local CP-maps}

The notion of minimality in Stinspring's theorem for CP-maps is very useful. To begin with it gives the uniqueness of the representation up to unitary equivalence. It was used by W. Arveson \cite{Arveson} to prove a Radon-Nikodym theorem for CP-maps. In this section, we identify the notion of minimality for the Stinespring's representation that is presented in Theorem \ref{Theorem: DosievStinespring}. 
Recall that $\mathcal{A}$ is a unital locally $C^{\ast}$-algebra and $\{\mathcal{H}; \mathcal{E}; \mathcal{D} \}$ denotes a quantized domain in a Hilbert space $\mathcal{H}$, where $\mathcal{E}=\{\mathcal{H}_{\alpha}: \alpha\in \Lambda\}.$  We use the notation $[M]$ to denote closure of the linear span of $M$ for any
subset $M$ of ${\mathcal H}.$

\begin{lemma}\label{Lemma: unionabsorption}
Let $\Big(\pi_{\varphi}, V_{\varphi}, \{\mathcal{H}^{\varphi}; \mathcal{E}^{\varphi}; \mathcal{D}^{\varphi}\} \Big)$ be a Stinespring representation of a map $\varphi$ in $\cpccloc$. Then
\begin{equation*}
    \overline{\bigcup\limits_{\ell \in \Omega}[\pi_{\varphi}(\mathcal{A})V_{\varphi}\mathcal{H}_{\ell}]} = [\pi_{\varphi}(\mathcal{A})V_{\varphi}\mathcal{D}].
\end{equation*}
\end{lemma}
\begin{proof}
This is clear as $\bigcup\limits_{l\in \Omega}{\mathcal H}_l= {\mathcal D}.$
\end{proof}

On first glance one would think that ${\mathcal H}^{\varphi }=  [\pi_{\varphi}(\mathcal{A})V_{\varphi}\mathcal{D}]$ could be the minimality condition for the Stinespring dilation. However the following definition seems to be more appropriate.

\begin{defn}\label{Definition: minimalityForDosiev}
	A Stinespring representation $\big({\pi}_{\varphi}, {V}_{\varphi}, \{\mathcal{H}^{\varphi}; \mathcal{E}^{\varphi}; \mathcal{D}^{\varphi}\}\big)$ of $\varphi$ is said to be  {\it minimal}, if  $\mathcal{H}_{\ell}^{\varphi} = [{{\pi}_{\varphi}(\mathcal{A}){V}_{\varphi}\mathcal{H}_{\ell}}]$, for every $\ell \in \Omega$.
\end{defn}
Note that by Lemma \ref{Lemma: unionabsorption}, under minimality it follows that $\mathcal{H}^{\varphi} = [\pi_{\varphi}(\mathcal{A})V_{\varphi}\mathcal{D}]$.
First  we see that any Stinespring representation can be modified to get a minimal Stinespring representation.
\begin{prop}\label{Proposition: minimality Dosiev} Let $\big(\pi_{\varphi}, V_{\varphi}, \{\mathcal{H}^{\varphi};\mathcal{E}^{\varphi}; \mathcal{D}^{\varphi}\}\big)$ be a Stinespring representation of the map $\varphi\in\cpccloc$ as in Theorem \ref{Theorem: DosievStinespring}. Then there is a minimal Stinespring representation $\big(\widetilde{\pi}_{\varphi}, \widetilde{V}_{\varphi}, \{\widetilde{\mathcal{H}}^{\varphi}; \widetilde{\mathcal{E}}^{\varphi}; \widetilde{\mathcal{D}}^{\varphi}\}\big)$ for $\varphi$ such that
\begin{equation*}
    \widetilde{\mathcal{E}}^{\varphi} \subseteq \mathcal{E}^{\varphi} \;~~ \text{and}\;~~ \widetilde{\mathcal{H}}^{\varphi} = [\widetilde{\pi}_{\varphi}(\mathcal{A})\widetilde{V}_{\varphi}\mathcal{D}].
\end{equation*}
\end{prop}
\begin{proof} Let us define $\widetilde{\mathcal{H}}_{\ell}^{\varphi}:= [\pi_{\varphi}(\mathcal{A})V_{\varphi}\mathcal{H}_{\ell}] $ for every $\ell \in \Omega$. It is clear that each $\widetilde{\mathcal{H}}_{\ell}^{\varphi}$ is a closed subspace of $\mathcal{H}_{\ell}^{\varphi}$ since $\pi_{\varphi}(a)\big|_{\mathcal{H}_{\ell}}\in \mathcal{B}(\mathcal{H}_{\ell})$ and $V_{\varphi}(\mathcal{H}_{\ell}) \subseteq \mathcal{H}^{\varphi}_{\ell}$ for every $a\in \mathcal{A}, \ell \in \Omega$. Moreover, $\widetilde{\mathcal{H}}_{\ell}^{\varphi} \subseteq \widetilde{\mathcal{H}}_{\ell^{\prime}}^{\varphi} $ whenever $\ell \leq \ell^{\prime}$. It implies that $\widetilde{\mathcal{E}}^{\varphi}:= \{\widetilde{\mathcal{H}}_{\ell}^{\varphi}:\; \ell \in \Omega\}$ is an upward filtered family of Hilbert spaces such that
$\widetilde{\mathcal{H}}_{\ell}^{\varphi}\subseteq {\mathcal{H}}_{\ell}^{\varphi}$ for every $\ell .$ In other words,
$\widetilde{\mathcal{E}}^{\varphi} \subseteq \mathcal{E}^{\varphi }$. Let $\widetilde{\mathcal{D}}^{\varphi}: = \bigcup\limits_{\ell \in \Omega} \widetilde{\mathcal{H}}_{\ell}^{\varphi}$ and $\widetilde{\mathcal{H}}^{\varphi}:= \overline{\widetilde{\mathcal{D}}^{\varphi}}$. Then $\big\{ \widetilde{\mathcal{H}}^{\varphi};\widetilde{\mathcal{E}}^{\varphi};\widetilde{\mathcal{D}}^{\varphi}\big\}$ is a quantized domain in the Hilbert space $\widetilde{\mathcal{H}}^{\varphi}$.

Further, by using the fact that $\widetilde{\mathcal{H}}_{\ell}^{\varphi}$ reduces every operator in $\pi_{\varphi}(\mathcal{A})$, we have
\begin{equation*}
\pi_{\varphi}(a)\big|_{\widetilde{\mathcal{H}}_{\ell}^{\varphi}} \in \mathcal{B}(\widetilde{\mathcal{H}}_{\ell}^{\varphi}),\; \text{for all}\; a \in \mathcal{A}.
\end{equation*}
Thus the map  $\widetilde{\pi}_{\varphi} \colon \mathcal{A}\to C^{\ast}_{\widetilde{\mathcal{E}}^{\varphi}}(\widetilde{\mathcal{D}}^{\varphi})$ defined by
\begin{equation*}
    \widetilde{\pi}_{\varphi}(a) ={\pi}_{\varphi}(a)\big|_{\widetilde{\mathcal{D}}^{\varphi}}, \; \text{for all}\; a \in \mathcal{A},
\end{equation*}
is well defined and is a unital contractive $\ast$-homomorphism. Let $\widetilde{V}_{\varphi}:= V_{\varphi}$. If $g, h \in \mathcal{D}$, then $g, h  \in \mathcal{H}_{\ell}$ for some $\ell \in \Omega$ and
\begin{align*}
    \Big\langle g, \widetilde{V}_{\varphi}^{\ast}\widetilde{\pi}_{\varphi}(a)\widetilde{V}_{\varphi}h\Big\rangle_{\mathcal{H}_{\ell}}= \Big\langle g, \varphi(a)h\Big\rangle_{\mathcal{H}_{\ell}}.
\end{align*}
This implies that $\varphi(a) \subseteq \widetilde{V}_{\varphi}^{\ast}\widetilde{\pi}_{\varphi}(a)\widetilde{V}_{\varphi}$, for all $a \in \mathcal{A}$. By Lemma \ref{Lemma: unionabsorption}, we have
\begin{equation*}
    \widetilde{\mathcal{H}}^{\varphi} = \overline{\widetilde{\mathcal{D}}^{\varphi}} = \overline{\bigcup\limits_{\ell \in \Omega}\widetilde{\mathcal{H}}_{\ell}^{\varphi}}=\overline{\bigcup\limits_{\ell \in \Omega}[\widetilde{\pi}_{\varphi}(\mathcal{A})\widetilde{V}_{\varphi}\mathcal{H}_{\ell}]} = [\widetilde{\pi}_{\varphi}(\mathcal{A})\widetilde{V}_{\varphi}\mathcal{D}]. \qedhere
\end{equation*}
\end{proof}
  Next, we show that the minimality condition in Definition \ref{Definition: minimalityForDosiev} is appropriate in the sense that any two minimal Stinespring representations are unitarily equivalent.
\begin{thm}\label{Theorem: UniquenessforDosiev} Let $ \big(\widetilde{\pi}_{\varphi}, \widetilde{V}_{\varphi}, \{\widetilde{\mathcal{H}}^{\varphi}; \widetilde{\mathcal{E}}^{\varphi}; \widetilde{D}^{\varphi}\}\big)$ and $\big(\widehat{\pi}_{\varphi}, \widehat{V}_{\varphi}, \{\widehat{\mathcal{H}}^{\varphi}; \widehat{\mathcal{E}}^{\varphi}; \widehat{D}^{\varphi}\}\big)$ be two minimal Stinespring representations of the map
$\varphi \in \cpccloc$. Then there is a unitary operator $U_{\varphi}: \widetilde{\mathcal{H}}^{\varphi} \to \widehat{\mathcal{H}}^{\varphi}$ such that
\begin{align}
U_{\varphi}\widetilde{V}_{\varphi}=\widehat{V}_{\varphi} \;\; \text{and}\;\; U_{\varphi}\widetilde{\pi}_{\varphi}(a)=  \widehat{\pi}_{\varphi}(a)U_{\varphi}, \;\text{for all}\; a\in \mathcal{A}.
\end{align}	
Equivalently, for every $a \in \mathcal{A}$, the following diagram commutes:
\begin{center}
    \begin{tikzcd}
    & & & &\widetilde{\mathcal{H}}^{\varphi} \ar[dddd, "U_{\varphi}", bend right=40]& \widetilde{\mathcal{H}}^{\varphi}\ar[dddd,"U_{\varphi}", bend left=40]& & & & \\
    & & & & \widetilde{\mathcal{D}}^{\varphi} \ar[u, hook]\ar[dd, "U_{\varphi}|_{\widetilde{\mathcal{D}}^{\varphi}}"]\ar[r, "\widetilde{\pi}_{\varphi}(a)"]& \ar[u, hook] \widetilde{\mathcal{D}}^{\varphi} \ar[dd, "U_{\varphi}|_{\widetilde{\mathcal{D}}^{\varphi}}"] & & & & \\
    \mathcal{H}\ar[uurrrr, "\widetilde{V}_{\varphi}"] \ar[ddrrrr, "\widehat{V}_{\varphi}"]& & & & & & & & & \mathcal{H}\ar[uullll, "\widetilde{V}_{\varphi}"] \ar[ddllll, "\widehat{V}_{\varphi}"]\\
    & & & & \widehat{\mathcal{D}}^{\varphi} \ar[d, hook] \ar[r, "\widehat{\pi}_{\varphi}(a)"]& \widehat{\mathcal{D}}^{\varphi} \ar[d, hook] & & & &\\
    & & & & \widehat{\mathcal{H}}^{\varphi} & \widehat{\mathcal{H}}^{\varphi} & & & &
    \end{tikzcd}
\end{center}
\end{thm}
\begin{proof} It is assumed that the Stinespring representations under consideration are minimal. So we have
\begin{equation*}
    \widetilde{\mathcal{H}}^{\varphi} = [\widetilde{\pi}_{\varphi}(\mathcal{A})\widetilde{V}_{\varphi}\mathcal{D}]\; \; \text{and}\;\; \widehat{\mathcal{H}}^{\varphi} = [\widehat{\pi}_{\varphi}(\mathcal{A})\widehat{V}_{\varphi}\mathcal{D}].
\end{equation*}
Define $U_{\varphi}\colon \text{span}\{\widetilde{\pi}_{\varphi}(\mathcal{A})\widetilde{V}_{\varphi}\mathcal{D}\} \to \text{span}\{\widehat{\pi}_{\varphi}(\mathcal{A})\widehat{V}_{\varphi}\mathcal{D}\}$ by
\begin{equation*}
    U_{\varphi}\Big(\sum\limits_{i=1}^{n} \widetilde{\pi}_{\varphi}(a_{i})\widetilde{V}_{\varphi}h_{i}\Big) = \sum\limits_{i=1}^{n} \widehat{\pi}_{\varphi}(a_{i})\widehat{V}_{\varphi}h_{i},
\end{equation*}
for every $a_{i}\in \mathcal{A}, h_{i}\in \mathcal{D}$, $i \in \{1,2,3,\cdots,n\}, n \in \mathbb{N}$. Then
\begin{align*}
    \Big\|U_{\varphi}\Big(\sum\limits_{i=1}^{n} \widetilde{\pi}_{\varphi}(a_{i})\widetilde{V}_{\varphi}h_{i}\Big)\Big\|^{2}_{\widehat{\mathcal{H}}_{\ell}^{\varphi}} &= \Big\|\sum\limits_{i=1}^{n} \widehat{\pi}_{\varphi}(a_{i})\widehat{V}_{\varphi}h_{i}\Big\|^{2}_{\widehat{\mathcal{H}}_{\ell}^{\varphi}}\\
    &= \sum\limits_{i,j=1}^{n}\Big\langle h_{i}, \; \widehat{V}_{\varphi}^{\ast}\widehat{\pi}_{\varphi}(a_{i})^{\ast}\widehat{\pi}_{\varphi}(a_{j})\widehat{V}_{\varphi}h_{j}\Big\rangle_{\mathcal{H}_{\ell}}\\
    &= \sum\limits_{i,j=1}^{n}\Big\langle h_{i}, \; \varphi(a_{i}^{\ast}a_{j})h_{j}\Big\rangle_{\mathcal{H}_{\ell}}\\
    &= \sum\limits_{i,j=1}^{n}\Big\langle h_{i}, \; \widetilde{V}_{\varphi}^{\ast}\widetilde{\pi}_{\varphi}(a_{i})^{\ast}\widetilde{\pi}_{\varphi}(a_{j})\widetilde{V}_{\varphi}h_{j}\Big\rangle_{\mathcal{H}_{\ell}}\\
    &= \Big\|\sum\limits_{i=1}^{n} \widetilde{\pi}_{\varphi}(a_{i})\widetilde{V}_{\varphi}h_{i}\Big\|^{2}_{\widetilde{\mathcal{H}}_{\ell}^{\varphi}},
\end{align*}
for every $\ell \in \Omega$. This implies that $U_{\varphi}$ is well defined and is an  isometry onto span$\{\widehat{\pi}_{\varphi}(\mathcal{A})\widehat{V}_{\varphi}\mathcal{D}\}$. Thus it can be extended to the whole of $\widetilde{\mathcal{H}}^{\varphi}$, denote this extension by itself. Note that $U_{\varphi}$ is a unitary since it is an isometry onto $\widehat{\mathcal{H}}^{\varphi}$. From the  definition of $U_{\varphi}$ and unitality of $\widetilde{\pi}_{\varphi}, \widehat{\pi}_{\varphi}$,
we see $U_{\varphi}\widetilde{V}_{\varphi} = \widehat{V}_{\varphi}$.

It  remains to show that $U_{\varphi}\widetilde{\pi}_{\varphi}(a) \subseteq  \widehat{\pi}_{\varphi}(a)U_{\varphi}$, for all $a \in \mathcal{A}$.
Now we observe that $$\text{dom}\big(U_{\varphi}\widetilde{\pi}_{\varphi}(a)\big)= \{\xi\in \widetilde{\mathcal{D}}:\widetilde{\pi}_{\varphi}(a) \xi \in \widetilde{\mathcal{H}}^{\varphi}\}= \widetilde{D}^{\varphi},$$
and
$$\text{dom}\big(\widehat{\pi}_{\varphi}(a)U_{\varphi}\big)= \{\xi\in \widetilde{\mathcal{H}}^{\varphi} \colon U_{\varphi}\xi \in  \widehat{\mathcal{D}}^{\varphi}\}$$
Since $U_{\varphi}$ is unitary, it implies that $\text{dom}\big(U_{\varphi}\widetilde{\pi}_{\varphi}(a)\big) = \text{dom}\big(\widehat{\pi}_{\varphi}(a)U_{\varphi}\big)$, for all $a \in \mathcal{A}$. Let $a, a_{i}\in \mathcal{A}$, $h_{i}\in \mathcal{H}_{\ell}$ for $i \in \{1,2,3,\cdots,n\}$, $n \in \mathbb{N}$. Then
\begin{align*}
    U_{\varphi}\widetilde{\pi}_{\varphi}(a)\Big(\sum\limits_{i=1}^{n}\widetilde{\pi}_{\varphi}(a_{i})\widetilde{V}_{\varphi}h_{i}\Big) &=U_{\varphi}\Big(\sum\limits_{i=1}^{n}\widetilde{\pi}_{\varphi}(a a_{i})\widetilde{V}_{\varphi}h_{i}\Big) \\
    &= \sum\limits_{i=1}^{n}\widehat{\pi}_{\varphi}(a a_{i})\widehat{V}_{\varphi}h_{i}\\
    &=\widehat{\pi}_{\varphi}(a)\Big( \sum\limits_{i=1}^{n}\widehat{\pi}_{\varphi}(a_{i})\widehat{V}_{\varphi}h_{i}\Big)\\
    &= \widehat{\pi}_{\varphi}(a)U_{\varphi}\Big( \sum\limits_{i=1}^{n}\widetilde{\pi}_{\varphi}(a_{i})\widetilde{V}_{\varphi}h_{i}\Big),
\end{align*}
for every $\ell \in \Omega$. Thus $U_{\varphi}\widetilde{\pi}_{\varphi}(a) = \widehat{\pi}_{\varphi}(a)U_{\varphi}$ on the dense set $\text{span}\{\widetilde{\pi}_{\varphi}(\mathcal{A})\widetilde{V}_{\varphi}\mathcal{D}\}$. Hence they are equal on $\widetilde{\mathcal{D}}^{\varphi}$. We conclude that $U_{\varphi}\widetilde{\pi}_{\varphi}(a) = \widehat{\pi}_{\varphi}(a)U_{\varphi}$, for all $a \in \mathcal{A}$.
\end{proof}

\section{Randon-Nikodym theorem for local CP-maps}
In this section, we present Radon-Nikodym theorem for local completely positive maps. We need the following  order relation on $\cpccloc$.
\begin{defn}\label{Definition: OrderOnCPCCloc}
Let $\psi, \varphi \in \cpccloc$. Then $\psi$ is said to be \emph{dominated} by $\varphi$, denoted by $\psi \leq \varphi$, if
\begin{equation*}
    \varphi-\psi \in \cpccloc.
\end{equation*}
\end{defn}
In this section we fix a  $\varphi \in \cpccloc$ and let $\Big(\pi_{\varphi}, V_{\varphi}, \{\mathcal{H}^{\varphi}; \mathcal{E}^{\varphi}; \mathcal{D}^{\varphi}\}\Big)$ be a minimal Stinespring representation. Firstly, the commutant of $\pi_{\varphi}(\mathcal{A})$ is denoted by $\pi_{\varphi}(\mathcal{A})^{\prime}$  and it is a collection of bounded operators defined as,
\begin{equation}
    \pi_{\varphi}(\mathcal{A})^{\prime}:= \big\{T \in \mathcal{B}(\mathcal{H}^{\varphi}):\; T\pi_{\varphi}(a) \subseteq \pi_{\varphi}(a)T, \; \text{for all}\; a \in \mathcal{A}\big\}.
\end{equation}
The following observation plays a crucial role in establishing our result.
\begin{lemma}\label{Lemma: minimalityforperp}
Let $\Big(\pi_{\varphi}, V_{\varphi}, \{\mathcal{H}^{\varphi}; \mathcal{E}^{\varphi}; \mathcal{D}^{\varphi}\}\Big)$ be a minimal Stinespring representation for a map $\varphi \in\cpccloc$. Then
\begin{equation*}
    \big[\pi_{\varphi}(\mathcal{A})V_{\varphi}(\mathcal{H}_{\ell}^{\bot}\cap \mathcal{D})\big] = (\mathcal{H}_{\ell}^{\varphi})^{\bot}, \; \text{for every}\; \ell \in \Omega.
\end{equation*}
\begin{proof}
Let us fix $\ell \in \Omega$ and let $\xi \in \text{span}\{\pi_{\varphi}(\mathcal{A})V_{\varphi}(\mathcal{H}_{\ell}^{\bot}\cap \mathcal{D})\}$. Then $\xi = \sum\limits_{i=1}^{n} \pi_{\varphi}(a_{i})V_{\varphi}h_{i}$, for some $a_{i}\in \mathcal{A}, h_{i}\in \mathcal{H}_{\ell}^{\bot}\cap \mathcal{D}$, $i\in \{1,2,3,\cdots, n\}$. Now let $\sum\limits_{j=1}^{m} \pi_{\varphi}(b_{j})V_{\varphi}g_{j}\in \Span\{\pi_{\varphi}(\mathcal{A})V_{\varphi}\mathcal{H}_{\ell}\},$ where
 $b_{j}\in \mathcal{A}$, $g_{j} \in \mathcal{H}_{\ell}$, $j\in \{1,2,3,\cdots,m\}$, we have
\begin{align*}
    \Big\langle \sum\limits_{i=1}^{n} \pi_{\varphi}(a_{i})V_{\varphi}h_{i},\; \sum\limits_{j=1}^{m} \pi_{\varphi}(b_{j})V_{\varphi}g_{j}\Big\rangle &= \sum\limits_{i =1}^{n} \sum\limits_{j=1}^{m}\Big\langle \pi_{\varphi}(a_{i})V_{\varphi}h_{i},\; \pi_{\varphi}(b_{j})V_{\varphi}g_{j}\Big\rangle\\
    &=  \sum\limits_{i =1}^{n} \sum\limits_{j=1}^{m}\Big\langle h_{i},\; V_{\varphi}^{\ast}\pi_{\varphi}(a_{i})^{\ast}\pi_{\varphi}(b_{j})V_{\varphi}g_{j}\Big\rangle\\
    &=\sum\limits_{i =1}^{n} \sum\limits_{j=1}^{m}\Big\langle h_{i},\; V_{\varphi}^{\ast}\pi_{\varphi}(a_{i}^{\ast}b_{j})V_{\varphi}g_{j}\Big\rangle\\
    &= \sum\limits_{i =1}^{n} \sum\limits_{j=1}^{m}\Big\langle h_{i},\; \varphi(a_{i}^{\ast}b_{j}) g_{j}\Big\rangle\\
    &= 0,
\end{align*}
due to the fact that $\varphi(a_{i}^{\ast}b_{j})g_{j} \in \mathcal{H}_{\ell}$ and $h_{i} \in \mathcal{H}_{\ell}^{\bot}$, for each $i,j$. It follows that $$\text{span}\{\pi_{\varphi}(\mathcal{A})V_{\varphi}(\mathcal{H}_{\ell}^{\bot}\cap \mathcal{D})\} \subseteq ( \Span\{\pi_{\varphi}(\mathcal{A})V_{\varphi}\mathcal{H}_{\ell}\})^{\bot}.$$
Thus by the minimality of Stinespring representation, we get that
\begin{equation*}
    \big[\pi_{\varphi}(\mathcal{A})V_{\varphi}(\mathcal{H}_{\ell}^{\bot}\cap \mathcal{D})\big] \subseteq (\mathcal{H}_{\ell}^{\varphi})^{\bot}.
\end{equation*}
 Suppose there is a $ 0\neq \eta \in \big[\pi_{\varphi}(\mathcal{A})V_{\varphi}(\mathcal{H}_{\ell}^{\bot}\cap \mathcal{D})\big]^{\bot} \cap (\mathcal{H}_{\ell}^{\varphi})^{\bot}$, then we see that
\begin{equation}\label{Equation: both}
    \langle \eta, \pi_{\varphi}(a)V_{\varphi}g \rangle  = 0 = \langle \eta, \pi_{\varphi}(a)V_{\varphi}h \rangle, \; \text{for every}\; a \in \mathcal{A}, g\in \mathcal{H}_{\ell}^{\bot}\cap \mathcal{D}, h\in \mathcal{H}_{\ell}.
\end{equation}
Let $x \in \mathcal{D}$. Then $x = g + h$ for some $g \in  \mathcal{H}_{\ell}, h\in \mathcal{H}_{\ell}^{\bot}\cap \mathcal{D}$. By Equation \ref{Equation: both}, we have $\langle \eta, \pi_{\varphi}(a)V_{\varphi}x \rangle = 0$, for every $a \in \mathcal{A}, x \in \mathcal{D}$.
Since $\mathcal{H}^{\varphi} = [\pi_{\varphi}(\mathcal{A})V_{\varphi}\mathcal{D}],$ thus $\langle \eta, \xi\rangle = 0$ for all $\xi \in
\mathcal{D}^{\varphi}$ and by the fact $\mathcal{D}^{\varphi}$ is dense in $\mathcal{H}^{\varphi}$, we obtain $\eta = 0$. This contradicts the fact
that $\eta \neq0.$ Hence we have
\begin{equation*}
    \big[\pi_{\varphi}(\mathcal{A})V_{\varphi}(\mathcal{H}_{\ell}^{\bot}\cap \mathcal{D})\big] = (\mathcal{H}_{\ell}^{\varphi})^{\bot}, \; \text{for every}\; \ell \in \Omega. \qedhere
\end{equation*}
\end{proof}
\end{lemma}
\begin{prop}\label{Proposition: phiTisCPandCC}
Let $\varphi \in \cpccloc$ and let $T \in \pi_{\mathcal{\varphi}}(\mathcal{A})^{\prime} \cap C^{\ast}_{\mathcal{E}^{\varphi}}(\mathcal{D}^{\varphi})$. Then the map $\rvarphi_{T} \colon \mathcal{A} \to C^{\ast}_{\mathcal{E}}(\mathcal{D})$ defined by
\begin{equation}\label{Equation: forphiT}
    \rvarphi_{T}(a) = V_{\varphi}^{\ast}T\pi_{\varphi}(a)V_{\varphi}\big|_{\mathcal{D}}, \; \text{for all}\; a \in \mathcal{A}
\end{equation}
is well-defined. Furthermore, $\varphi_{T} \in \cpccloc$ and $ \rvarphi_{T} \leq \varphi$ whenever $0 \leq T \leq I$.
\end{prop}
\begin{proof}
First we show that $\rvarphi_{T}(a) \in C^{\ast}_{\mathcal{E}}(\mathcal{D})$, for all $a \in \mathcal{A}$. Let $h \in \mathcal{H}_{\ell}$ and $g \in \mathcal{H}_{\ell}^{\bot} \cap \mathcal{D}$  for some $\ell \in \Omega$. Then $V_{\varphi}(g) = \pi_{\varphi}(1)V_{\varphi}(g) \in (\mathcal{H}_{\ell}^{\varphi})^{\bot}$ by Lemma \ref{Lemma: minimalityforperp}. Thus $V_{\varphi}g \in (\mathcal{H}_{\ell}^{\varphi})^{\bot}$  and  $T\pi_{\varphi}(a)V_{\varphi}h \in \mathcal{H}_{\ell}^{\varphi}$   since $ \pi_{\varphi}(a)\in C^{\ast}_{\mathcal{E}^{\varphi}}(\mathcal{D}^{\varphi})$ for all $a \in \mathcal{A}$. Therefore,
$$    \big\langle g,\; \varphi_{T}(a)h\big\rangle = \big\langle g,\; V_{\varphi}^{\ast}T\pi_{\varphi}(a)V_{\varphi}h\big\rangle\\
    =  \big\langle V_{\varphi}g,\; {T}\pi_{\varphi}(a)V_{\varphi}h\big\rangle
    =0,$$
for all $a \in \mathcal{A},\; h \in \mathcal{H}_{\ell},\; g \in \mathcal{H}_{\ell}^{\bot}\cap \mathcal{D}$. It is clear that $\mathcal{H}_{\ell}$, $\mathcal{H}_{\ell}^{\bot}\cap \mathcal{D}$ are invariant under $\rvarphi_{T}(a)$ such that
$\rvarphi_{T}(a)\big|_{\mathcal{H}_{\ell}} = V_{\varphi}^{\ast}T\pi_{\varphi}(a)V_{\varphi}\big|_{\mathcal{H}_{\ell}} \in \mathcal{B}(\mathcal{H}_{\ell})$ for all $ a \in \mathcal{A}, \ell \in \Omega $. Hence $\rvarphi_{T}$ is a well-defined map.

 Suppose $ 0 \leq T \leq I$. Now we  show that  $\rvarphi_{T} \in \cpccloc$. Let $A: = [a_{ij}]_{n\times n} \in M_{n}(\mathcal{A})$ and $h_{1}, h_{2}, h_{3}, \cdots h_{n} \in \mathcal{H}_{\ell}$ for some $\ell \in \Omega$. Then
\begin{align*}
    \Big\langle \begin{bmatrix}
    h_{1}\\
   \vdots\\
    h_{n}
    \end{bmatrix}, \;\rvarphi_{T}^{(n)}(A)\begin{bmatrix}
    h_{1}\\
   \vdots\\
    h_{n}
    \end{bmatrix}
    \Big\rangle_{\mathcal{H}_{\ell}^{n}}
    &= \Big\langle \begin{bmatrix}
    h_{1}\\
    \vdots\\
    h_{n}
    \end{bmatrix}, \;[\rvarphi_{T}(a_{ij})]_{n\times n}\begin{bmatrix}
    h_{1}\\
    \vdots\\
    h_{n}
    \end{bmatrix}
    \Big\rangle_{\mathcal{H}_{\ell}^{n}}\\
    &= \Big\langle \begin{bmatrix}
    h_{1}\\
    \vdots\\
    h_{n}
    \end{bmatrix}, \;[V_{\varphi}^{\ast}T\pi_{\varphi}(a_{ij})V_{\varphi}]_{n\times n}\begin{bmatrix}
    h_{1}\\
    \vdots\\
    h_{n}
    \end{bmatrix}
    \Big\rangle_{\mathcal{H}_{\ell}^{n}}\\
    &= \Big\langle \begin{bmatrix}
    h_{1}\\
    \vdots\\
    h_{n}
    \end{bmatrix}, \;[V_{\varphi}^{\ast}\sqrt{T}\pi_{\varphi}(a_{ij})\sqrt{T}V_{\varphi}]_{n\times n}\begin{bmatrix}
    h_{1}\\
    \vdots\\
    h_{n}
    \end{bmatrix}
    \Big\rangle_{\mathcal{H}_{\ell}^{n}}\\
    &= \Big\langle \begin{bmatrix}
    \sqrt{T}V_{\varphi}h_{1}\\
    \vdots\\
    \sqrt{T}V_{\varphi}h_{n}
    \end{bmatrix}, \;[\pi_{\varphi}(a_{ij})]_{n\times n}\begin{bmatrix}
    \sqrt{T}V_{\varphi}h_{1}\\
    \vdots\\
    \sqrt{T}V_{\varphi}h_{n}
    \end{bmatrix}
    \Big\rangle_{(\mathcal{H}_{\ell}^{\varphi})^{n}}\\
    &= \Big\langle \begin{bmatrix}
    \sqrt{T}V_{\varphi}h_{1}\\
    \vdots\\
    \sqrt{T}V_{\varphi}h_{n}
    \end{bmatrix}, \;\pi_{\varphi}^{(n)}(A)\begin{bmatrix}
    \sqrt{T}V_{\varphi}h_{1}\\
    \vdots\\
    \sqrt{T}V_{\varphi}h_{n}
    \end{bmatrix}
    \Big\rangle_{(\mathcal{H}_{\ell}^{\varphi})^{n}}.
    \end{align*}
Since $\pi_{\varphi} \in \mathcal{CP}_{loc}(\mathcal{A}, C^{\ast}_{\mathcal{E}^{\varphi}}\big(\mathcal{D}^{\varphi})\big)$, we know that for $\ell \in \Omega$ there exists $\alpha \in \Lambda$ such that $\pi^{(n)}_{\varphi}(A)\big|_{\mathcal{H}_{\ell}^{\varphi}} \geq 0$ whenever $A \geq_{\alpha} 0$ for all $n.$  In particular,

\begin{align*}
\Big\langle \begin{bmatrix}
    h_{1}\\
   \vdots\\
    h_{n}
    \end{bmatrix}, \;\rvarphi_{T}^{(n)}(A)\begin{bmatrix}
    h_{1}\\
   \vdots\\
    h_{n}
    \end{bmatrix}
    \Big\rangle_{\mathcal{H}_{\ell}^{n}}
    =
   \Big\langle\begin{bmatrix} \sqrt{T}V_{\varphi}h_{1}\\
    \vdots\\
 \sqrt{T}V_{\varphi}h_{n}
\end{bmatrix},\; \pi_{\varphi}^{(n)} (A)\begin{bmatrix}
\sqrt{T}V_{\varphi}h_{1}\\
\vdots\\
\sqrt{T}V_{\varphi}h_{n}
\end{bmatrix}\Big\rangle_{\mathcal{H}_{\ell}^{n}}
\geq 0,
\end{align*}
whenever $A \geq_{\alpha} 0$.
This implies that there exists $\alpha \in \Lambda $ such that for every $n \in \mathbb{N}$ we have
\begin{equation*}
    \varphi_{T}^{(n)} (A)\big|_{\mathcal{H}_{\ell}^{n}}\geq 0,\; \text{whenever}\; A \geq_{\alpha} 0.
\end{equation*}
As a result $\rvarphi_{T} \in \mathcal{CP}_{loc}\big(\mathcal{A}, C^{\ast}_{\mathcal{E}}(\mathcal{D})\big)$. Now by using the fact that $T$ and $V_{\varphi}$ are contractions, we show $\rvarphi_{T}\in \mathcal{CC}_{loc}\big(\mathcal{A}, C^{\ast}_{\mathcal{E}}(\mathcal{D})\big)$ as follows:
\begin{align*}
    \Big\|\rvarphi_{T}^{(n)}(A)\begin{bmatrix}
    h_{1}\\
    \vdots\\
    h_{n}\end{bmatrix}\Big\|^{2}_{\mathcal{H}_{\ell}^{n}}
    &= \Big\|\big[\rvarphi_{T}(a_{i,j})\big]_{n \times n}\begin{bmatrix}
    h_{1}\\
    \vdots\\
    h_{n}\end{bmatrix}\Big\|^{2}_{\mathcal{H}_{\ell}^{n}}\\
    &= \sum\limits_{i=1}^{n}\; \Big\| \sum\limits_{j=1}^{n} \rvarphi_{T}(a_{ij})h_{j}\Big\|^{2}_{\mathcal{H}_{\ell}} \\
    &= \sum\limits_{i=1}^{n}\;\Big\| \sum\limits_{j=1}^{n} V_{\varphi}^{\ast}T\pi_{\varphi}(a_{ij})V_{\varphi}h_{j}\Big\|^{2}_{\mathcal{H}_{\ell}}\\
    &= \sum\limits_{i=1}^{n}\;\Big\|  V_{\varphi}^{\ast}T \Big(\sum\limits_{j=1}^{n}\pi_{\varphi}(a_{ij})V_{\varphi}h_{j}\Big)\Big\|^{2}_{\mathcal{H}_{\ell}}\\
    &\leq \sum\limits_{i=1}^{n}\;\Big\| \sum\limits_{j=1}^{n} \pi_{\varphi}(a_{ij})V_{\varphi}h_{j}\Big\|^{2}_{\mathcal{H}_{\ell}^{\varphi}}\\
    &= \Big\|\pi_{\varphi}^{(n)}(A)\begin{bmatrix}
    V_{\varphi}h_{1}\\
    \vdots\\
    V_{\varphi}h_{n}\end{bmatrix}\Big\|^{2}_{(\mathcal{H}_{\ell}^{\varphi})^{n}}.
    \end{align*}
    Since $\pi_{\varphi}$ is a locally completely contractive map, for $\ell \in \Omega$, there exists an $\alpha \in \Lambda$ such that
    \begin{equation*}
        \Big\| \pi_{\varphi}^{(n)}(B)\Big\|_{(\mathcal{H}_{\ell}^{\varphi})^{n}} \leq \|B\|_{\alpha}, \; \text{for every }\; B \in M_{n}(\mathcal{A}), n\in \mathbb{N}.
    \end{equation*}
    This implies that
$$        \Big\|\rvarphi_{T}^{(n)}(A)\begin{bmatrix}
    h_{1}\\
    \vdots\\
    h_{n}\end{bmatrix}\Big\|^{2}_{\mathcal{H}_{\ell}^{n}}
    \leq \Big\|\pi_{\varphi}^{(n)}(A)\begin{bmatrix}
    V_{\varphi}h_{1}\\
    \vdots\\
    V_{\varphi}h_{n}\end{bmatrix}\Big\|^{2}_{(\mathcal{H}_{\ell}^{\varphi})^{n}} 
    \leq \|A\|_{\alpha}\; \Big\|\begin{bmatrix}
    V_{\varphi}h_{1}\\
    \vdots\\
    V_{\varphi}h_{n}\end{bmatrix}\Big\|^{2}_{(\mathcal{H}_{\ell}^{\varphi})^{n}}
    \leq \|A\|_{\alpha}\; \Big\|\begin{bmatrix}
    h_{1}\\
    \vdots\\
    h_{n}\end{bmatrix}\Big\|^{2}_{\mathcal{H}_{\ell}^{n}},
   $$
    for every $n \in \mathbb{N}$.    Therefore, $\rvarphi_{T} \in \cpccloc$ whenever $0 \leq T \leq I$.
Moreover, by the definition of $\rvarphi_{T}$, we see that
$$    (\varphi- \rvarphi_{T})(a) = V_{\varphi}^{\ast}\pi_{\varphi}(a)V_{\varphi}\big|_{\mathcal{D}} - V_{\varphi}^{\ast}T\pi_{\varphi}(a)V_{\varphi}\big|_{\mathcal{D}}
    = V_{\varphi}^{\ast}(I-T)\pi_{\varphi}(a)V_{\varphi}\big|_{\mathcal{D}}
    = \rvarphi_{(I-T)}(a),$$
for all $a \in \mathcal{A}$. We know that $\rvarphi_{(I-T)} \in \cpccloc$ since $0 \leq (I-T) \leq I$ and by Definition \ref{Definition: OrderOnCPCCloc}, we have $\rvarphi_{T}\leq \varphi$. Hence the result.
    \end{proof}
\begin{lemma}\label{Lemma: uniqueness}
As above consider  $\rvarphi_{T} $ for $T \in \pi_{\varphi}(\mathcal{A})^{\prime} \cap C^{\ast}_{\mathcal{E}^{\varphi}}(\mathcal{D}^{\varphi}).$ Then  $\rvarphi_{T}=0$ if
and only if  $T = 0$.
\end{lemma}
\begin{proof} Assume $\rvarphi_{T} =0.$
 Let $a_{i}, b_{j}\in \mathcal{A},\; g_{i}, h_{j}\in \mathcal{D}$, for $ i \in \{ 1, 2,3\cdots m\}$ and $ j \in \{ 1,2,3 \cdots n\}$, where $m,n \in \mathbb{N}$. Then
\begin{align*}
   \Big\langle \sum\limits_{j=1}^{n}\pi_{\varphi}(b_{j})V_{\varphi}h_{j},\; T\Big(\sum\limits_{i=1}^{m}\pi_{\varphi}(a_{i})V_{\varphi}g_{i}\Big)\Big\rangle &
   = \sum\limits_{j=1}^{n}\sum\limits_{i=1}^{m} \Big\langle \pi_{\varphi}(b_{j})V_{\varphi}h_{j},\; T\pi_{\varphi}(a_{i})V_{\varphi}g_{i}\Big\rangle\\
   &= \sum\limits_{j=1}^{n}\sum\limits_{i=1}^{m} \Big\langle h_{j},\; V_{\varphi}^{\ast}T\pi_{\varphi}(b_{j}^{\ast}a_{i})V_{\varphi}g_{i} \Big\rangle\\
   &= \sum\limits_{j=1}^{n}\sum\limits_{i=1}^{m} \Big\langle h_{j},\; \rvarphi_{T}(b_{j}^{\ast}a_{i})g_{i}\Big\rangle\\
   &=0.
\end{align*}
By the minimality of Stinespring representation, this implies that $T = 0$. The converse is trivial.
\end{proof}
We may think of $T$ as the Radon-Nikodym derivative of the map $\rvarphi_{T}$ with respect to $\varphi$. We consolidate our results in the following Theorem.
\begin{thm}\label{Theorem: MainTheroem1}
Let $\varphi, \psi \in \cpccloc$. Then $\psi \leq \varphi$ if and only if  there exists a unique  $T\in \pi_{\varphi}(\mathcal{A})^{\prime} \cap C^{\ast}_{\mathcal{E}^{\varphi}}(\mathcal{D}^{\varphi})$ such that $0 \leq T\leq I$ and 
\begin{equation}\label{RNTheorem}
    \psi(a) \subseteq  V_{\varphi}^{\ast}T\pi_{\varphi}(a)V_{\varphi},\; \text{for all}\; a \in \mathcal{A}.
\end{equation}
\end{thm}
\begin{proof} Suppose that  $T \in \pi_{\varphi}(\mathcal{A})^{\prime} \cap C^{\ast}_{\mathcal{E}^{\varphi}}(\mathcal{D}^{\varphi})$ be such that $0 \leq T\leq I$ and
$\psi $ is defined by equation \ref{RNTheorem},  then by Proposition \ref{Proposition: phiTisCPandCC}, it is clear that $\psi \leq \varphi$.
Now we prove the converse. Suppose that $\psi \leq \varphi$.  Let us define a map $S$ on the dense subset $\Span\{\pi_{\varphi}(\mathcal{A})V_{\varphi}(\mathcal{D})\}$ of $\mathcal{H}^{\varphi}$ by
\begin{equation}\label{Equation: Tphi}
    S\Big(\sum\limits_{i=1}^{n}\pi_{\varphi}(a_{i})V_{\varphi}h_{i}\Big) = \sum\limits_{i=1}^{n}\pi_{\psi}(a_{i})V_{\psi}h_{i},
    \end{equation}
    for $a_{i}\in \mathcal{A}, h_{i} \in \mathcal{D}, i = 1,2,3,\ldots , n, n \geq 1$. Since $\mathcal{D}$ is the union of $\mathcal { H}_{\ell }$'s, and we have a finite collection of $h_i$'s,  there exists an $\ell \in \Omega$ such that $h_{i} \in \mathcal{H}_{\ell}$ for every $i$. Further, by using the fact that  $V_{\varphi} (\mathcal{H}_{\ell}) \subseteq \mathcal{H}_{\ell}^{\varphi}, \; V_{\psi} (\mathcal{H}_{\ell}) \subseteq \mathcal{H}_{\ell}^{\psi}, \; \pi_{\varphi}(a_{i})\big|_{\mathcal{H}_{\ell}^{\varphi}} \in \mathcal{B}(\mathcal{H}_{\ell}^{\varphi})$ and $\pi_{\psi}(a_{i})\big|_{\mathcal{H}_{\ell}^{\psi}} \in \mathcal{B}(\mathcal{H}_{\ell}^{\psi}) $ for $i = 1,2,3 \ldots , n$, we see that
    \begin{align*}
        \Big\|S\Big(\sum\limits_{i=1}^{n}\pi_{\varphi}(a_{i})V_{\varphi}h_{i}\Big)\Big\|^{2}_{\mathcal{H}_{\ell}^{\psi}} &= \Big\|\sum\limits_{i=1}^{n}\pi_{\psi}(a_{i})V_{\psi}h_{i} \Big\|^{2}_{\mathcal{H}_{\ell}^{\psi}}\\
        &= \sum\limits_{i,j=1}^{n}\big\langle \pi_{\psi}(a_{i})V_{\psi}h_{i}, \pi_{\psi}(a_{j})V_{\psi}h_{j} \big\rangle_{\mathcal{H}_{\ell}^{\psi}}\\
        &= \sum\limits_{i,j=1}^{n}\big\langle h_{i}, V_{\psi}^{\ast}\pi_{\psi}(a_{i}^{\ast}a_{j})V_{\psi}h_{j} \big\rangle_{\mathcal{H}_{\ell}}\\
        &= \sum\limits_{i,j=1}^{n}\big\langle h_{i}, \psi(a_{i}^{\ast}a_{j})h_{j} \big\rangle_{\mathcal{H}_{\ell}}\\
        &= \Big\langle \begin{bmatrix}
        h_{1}\\
        \vdots\\
        h_{n}
        \end{bmatrix}, \; \psi^{(n)}\Big(\big[a_{i}^{\ast}a_{j}\big]_{n\times n}\Big)\begin{bmatrix}
        h_{1}\\
        \vdots\\
        h_{n}
        \end{bmatrix}
        \Big\rangle_{\mathcal{H}_{\ell}^{n}}.
    \end{align*}
    That is,
    \begin{equation}\label{Equation: Tleq}
        \Big\|S\Big(\sum\limits_{i=1}^{n}\pi_{\varphi}(a_{i})V_{\varphi}h_{i}\Big)\Big\|^{2}_{\mathcal{H}_{\ell}^{\psi}} = \Big\langle \begin{bmatrix}
        h_{1}\\
        \vdots\\
        h_{n}
        \end{bmatrix}, \; \psi^{(n)}\Big(\big[a_{i}^{\ast}a_{j}\big]_{n\times n}\Big)\begin{bmatrix}
        h_{1}\\
        \vdots\\
        h_{n}
        \end{bmatrix}
        \Big\rangle_{\mathcal{H}_{\ell}^{n}}.
    \end{equation}
    From the assumption of $\psi \leq \varphi$, it follows that for each $\ell \in \Omega$, there exist $\alpha \in \Lambda$ such that
    \begin{equation*}
        \varphi^{(n)}(A)- \psi^{(n)} (A)\big|_{\mathcal{H}_{\ell}^{n}} \geq 0, \; \text{ whenever}\; A \geq_{\alpha} 0\;\text{in}\; M_{n}(\mathcal{A}),
    \end{equation*}
    for every $n \in \mathbb{N}$. In particular, we have $[a_{i}^{\ast}a_{j}]_{n\times n} \geq_{\alpha} 0$ in $M_{n}(\mathcal{A})$, for every
    $\alpha \in \Omega, n \in \mathbb{N}$. This implies that
    \begin{equation}\label{Equation: psinleqvarphin}
    \Big\langle \begin{bmatrix}
        h_{1}\\
        \vdots\\
        h_{n}
        \end{bmatrix},  \; \psi^{(n)}\Big(\big[a_{i}^{\ast}a_{j}\big]_{n\times n}\Big)\begin{bmatrix}
        h_{1}\\
        \vdots\\
        h_{n}
        \end{bmatrix}
        \Big\rangle_{\mathcal{H}_{\ell}^{n}}
     \leq \Big \langle \begin{bmatrix}
        h_{1}\\
        \vdots\\
        h_{n}
        \end{bmatrix}, \; \varphi^{(n)}\Big(\big[a_{i}^{\ast}a_{j}\big]_{n\times n}\Big)\begin{bmatrix}
        h_{1}\\
        \vdots\\
        h_{n}
        \end{bmatrix}
        \Big \rangle_{\mathcal{H}_{\ell}^{n}}.
    \end{equation}
      Thus by Equations (\ref{Equation: Tleq}), (\ref{Equation: psinleqvarphin}), we get that
    \begin{align*}
        \Big\|S\Big(\sum\limits_{i=1}^{n}\pi_{\varphi}(a_{i})V_{\varphi}h_{i}\Big)\Big\|^{2}_{\mathcal{H}_{\ell}^{\psi}} &\leq \Big \langle \begin{bmatrix}
        h_{1}\\
        \vdots\\
        h_{n}
        \end{bmatrix}, \; \varphi^{(n)}\Big(\big[a_{i}^{\ast}a_{j}\big]_{n\times n}\Big)\begin{bmatrix}
        h_{1}\\
        \vdots\\
        h_{n}
        \end{bmatrix} \Big\rangle_{\mathcal{H}_{\ell}^{n}}\\
        &\leq \sum\limits_{i,j=1}^{n}\big\langle h_{i}, V_{\varphi}^{\ast}\pi_{\varphi}(a_{i}^{\ast}a_{j})V_{\varphi}h_{j} \big\rangle_{\mathcal{H}_{\ell}}\\
        &= \sum\limits_{i,j=1}^{n}\big\langle \pi_{\varphi}(a_{i})V_{\varphi}h_{i}, \pi_{\varphi}(a_{j})V_{\varphi}h_{j} \big\rangle_{\mathcal{H}_{\ell}^{\varphi}}\\
        &= \Big\|\sum\limits_{i=1}^{n}\pi_{\varphi}(a_{i})V_{\varphi}h_{i}\Big\|^{2}_{\mathcal{H}_{\ell}^{\varphi}},
    \end{align*}
    for every $\ell \in \Omega, n \in \mathbb{N}$. This shows that $S$ is well defined and is a contraction on the dense set $\Span\{\pi_{\varphi}(\mathcal{A})V_{\varphi}(\mathcal{D})\}$. Hence it can be extended to whole of $\mathcal{H}^{\varphi}$ and again we denote it by $S$. Then $S \in \mathcal{B}(\mathcal{H}^{\varphi}, \mathcal{H}^{\psi})$ is a contraction. Moreover, by the definition of $S$ and above observations, it follows that
    \begin{equation}\label{Equation: InvarianceOfT}
    S({\mathcal{H}^{\varphi}_{\ell}}) \subseteq {\mathcal{H}^{\psi}_{\ell}} \; \text{ and}\;  S\big|_{{\mathcal{H}^{\varphi}_{\ell}}} \in \mathcal{B}({\mathcal{H}^{\varphi}_{\ell}}, {\mathcal{H}^{\psi}_{\ell}}),\; \text{ for every} \; \ell \in \Omega.
    \end{equation}
    Let us take $T= S^{\ast}S$. It is clear that $0 \leq T \leq I$ and $T \in C^{\ast}_{\mathcal{E}^{\varphi}}(\mathcal{D}^{\varphi})$ by Equation (\ref{Equation: InvarianceOfT}). First, we show that $T \in\pi_{\varphi}(\mathcal{A})^{\prime}$ as follows:
    Let $a, a_{i} \in \mathcal{A}, h_{i}\in \mathcal{D}, i \in \{ 1,2,3 \hdots n\}$, where $n\in \mathbb{N}$. Then
\begin{align*}
    S\pi_{\varphi}(a) \Big(\sum\limits_{i = 1}^{n}\pi_{\varphi}(a_{i})V_{\varphi}h_{i}\Big)&= S \Big(\sum\limits_{i = 1}^{n}\pi_{\varphi}(aa_{i})V_{\varphi}h_{i}\Big)\\
    &= \sum\limits_{i=1}^{n} \pi_{\psi}(aa_{i})V_{\psi}h_{i}\\
    &= \pi_{\psi}(a)\Big(\sum\limits_{i=1}^{n} \pi_{\psi}(a_{i})V_{\psi}h_{i}\Big)\\
    &= \pi_{\psi}(a)S\Big(\sum\limits_{i=1}^{n} \pi_{\varphi}(a_{i})V_{\varphi}h_{i}\Big).
\end{align*}
By the continuity of $S$, we get that $S\pi_{\varphi}(a) = \pi_{\psi}(a)S$ on $\text{dom}(S\pi_{\varphi}(a)) = \mathcal{D}$, for all $a \in \mathcal{A}$. Since $\mathcal{D}\subseteq \text{dom} (\pi_{\psi}(a)S) = \{x \in \mathcal{H}^{\varphi}:\; Sx \in \mathcal{D}^{\psi}\}$, we have
\begin{equation}\label{Equation: TpiphicontainedinpipsiT}
    S\pi_{\varphi}(a) \subseteq \pi_{\psi}(a)S, \; \text{for all}\; a \in \mathcal{A}.
\end{equation}
Since $\pi_{\varphi}$ and $\pi_{\psi}$ are $\ast$-homomorphisms, by taking adjoints on both side of Equation (\ref{Equation: TpiphicontainedinpipsiT}), we have
\begin{equation}\label{Equation: adjointofabove}
    S^{\ast}\pi_{\psi}(a) \subseteq \pi_{\varphi}(a)S^{\ast}, \; \text{for all}\; a \in \mathcal{A}.
\end{equation}
From Equations (\ref{Equation: TpiphicontainedinpipsiT}), (\ref{Equation: adjointofabove}) it follows that
\begin{equation*}
     T\pi_{\varphi}(a) = S^{\ast}S\pi_{\varphi}(a)\subseteq S^{\ast}\pi_{\psi}(a)S\subseteq \pi_{\varphi}(a)S^{\ast}S = \pi_{\varphi}(a)T,
\end{equation*}
for all $a \in \mathcal{A}$. Equivalently, $T \in \pi_{\varphi}(\mathcal{A})^{\prime}$. Next, for every $a \in \mathcal{A}$ and $ h,k \in \mathcal{D}$, we have
\begin{align*}
    \Big\langle k, \; V_{\varphi}^{\ast}T\pi_{\varphi}(a)V_{\varphi}h\Big\rangle
    &= \Big\langle k, \; V_{\varphi}^{\ast}S^{\ast}S\pi_{\varphi}(a)V_{\varphi}h\Big\rangle\\
    &=\Big\langle SV_{\varphi}k, \;S\pi_{\varphi}(a)V_{\varphi}h\Big\rangle\\
    &= \Big\langle V_{\psi}k, \;\pi_{\psi}(a)V_{\psi}h\Big\rangle\\
    &= \Big\langle k, \; V_{\psi}^{\ast}\pi_{\psi}(a)V_{\psi}h\Big\rangle\\
    &= \Big\langle k, \; V_{\psi}^{\ast}\pi_{\psi}(a)V_{\psi}h\Big\rangle\\
    &=  \Big\langle k, \; \psi(a)h\Big\rangle.
\end{align*}
Therefore, $\psi(a) \subseteq  V_{\varphi}^{\ast}T\pi_{\varphi}(a)V_{\varphi}$, for all $a \in \mathcal{A}$. Furthermore, the uniqueness of $T$ follows from Lemma \ref{Lemma: uniqueness}. Hence the result.
\end{proof}
\begin{cor}\label{isomorphism} The map $\zeta \colon \big\{T \in \pi_{\varphi}(\mathcal{A})^{\prime}\cap C^{\ast}_{\mathcal{E}^{\varphi}}(\mathcal{D}^{\varphi}):\; 0 \leq T \leq I\big\} \longrightarrow
\{ \psi \in \cpccloc : 0\leq \psi \leq \varphi \} $
defined by
\begin{equation*}
    \zeta (T) = \rvarphi_{T},
\end{equation*} is an order isomorphism preserving convexity structure:
\begin{equation*}
    \zeta (pT_1+ (1-p)T_2)=p\zeta (T_1)+(1-p)\zeta (T_2)
\end{equation*} for $0\leq p\leq 1$ and $0\leq T_1, T_2\leq I$ in $ \pi_{\varphi}(\mathcal{A})^{\prime}\cap C^{\ast}_{\mathcal{E}^{\varphi}}(\mathcal{D}^{\varphi}) .$
\end{cor}
\begin{proof} By Lemma \ref{Lemma: uniqueness} and Theorem \ref{Theorem: MainTheroem1}, it follows that the map $\zeta$ is bijective.  Next, we show that $\zeta$ is order preserving. Let $0\leq T \leq S\leq I$. Then $(S-T) \in \pi_{\varphi}(\mathcal{A})^{\prime}\cap C^{\ast}_{\mathcal{E}^{\varphi}}(\mathcal{D}^{\varphi})$, $0 \leq (S-T) \leq I$ and by Proposition \ref{Proposition: phiTisCPandCC}, we have
\begin{equation*}
    \zeta(S)-\zeta(T) = \rvarphi_{S} - \rvarphi_{T}= \rvarphi_{(S-T)} \in \cpccloc.
\end{equation*}
This shows that $\zeta(T) \leq \zeta(S)$. The second part is obvious from the definition of $\zeta .$
\end{proof}
\section{Stinespring's representation for local CP-inducing maps}
Now we turn our discussion to local CP-inducing maps on Hilbert modules over locally $C^{\ast}$-algebras. In this section, we prove Stinespring type representation for such maps. Let us recall the notion of local CP-inducing map on Hilbert modules over locally $C^{\ast}$-algebra (see Definition \ref{Definition: innerproductAmodule}). Let $E$ be a Hilbert module over a locally $C^{\ast}$-algebra $\mathcal{A}$ and let $\big\{\mathcal{H};\mathcal{E};\mathcal{D}\}$, $\{\mathcal{K}; \mathcal{F};\mathcal{O}\big\}$ be a quantized domains. A map $\Phi \colon E \to C^{\ast}_{\mathcal{E}, \mathcal{F}}(\mathcal{D}, \mathcal{O})$ is said to be \emph{local CP-inducing map} if there is a $\varphi \in \cpccloc$ such that
\begin{equation*}
    \big\langle \Phi(x), \Phi(y)\big\rangle = \varphi(\langle x, y\rangle), \; \text{for all}\; x,y \in E.
\end{equation*}
 Let us denote the class of local CP-inducing maps on $E$ by $C_{loc}(E, C^{*}_{\mathcal{E, F}}(\mathcal{D, O}))$.
From now on wards, we denote the local CP-inducing map $\Phi$ along with the associated $\varphi \in \cpccloc$ by the pair $(\varphi, \Phi)$. In this section, we prove Stinespring's theorem for local CP-inducing maps $(\varphi, \Phi)$ and discuss the minimality of such representations. This result can be seen as  a generalization of \cite[Theorem 2.1]{Bhat}.

The following lemma plays a key role in proving our result and its advantage may be noticed repeatedly.  Throughout this section, $\mathcal{A}$ denotes a  unital locally $C^{\ast}$-algebra and $E$ denotes a  Hilbert  $\mathcal{A}$-module.
\begin{lemma}\label{Lemma: Denseness}
     Let $\{\mathcal{H};\mathcal{E};\mathcal{D}\}$ and $\{\mathcal{K};\mathcal{F};\mathcal{O}\}$ be quantized domains in $\mathcal{H}$ and $\mathcal{K}$ respectively. Then for any complex linear map $\Phi \colon E \to C^{\ast}_{\mathcal{E}, \mathcal{F}}(\mathcal{D}, \mathcal{O})$ we have
     \begin{equation*}
         \overline{\bigcup\limits_{\ell \in \Omega}[\Phi(E)\mathcal{H}_{\ell}]} = [\Phi(E)\mathcal{D}].
     \end{equation*}
 \end{lemma}
 \begin{proof} This is clear as $\mathcal{D}= \bigcup\limits_{\ell \in \Omega }{\mathcal H}_{\ell }.$
 \end{proof}
\begin{thm}\label{Theorem: Stinespringforphimaps}
Let  $\Phi \colon E \to C^{\ast}_{\mathcal{E}, \mathcal{F}}(\mathcal{D}, \mathcal{O})$ be a local CP-inducing map associated to  $\varphi \in \cpccloc$. Then there exists a pair of triples
     \begin{equation*}
         \Big( \big(\pi_{\varphi}, V_{\varphi}, \{\mathcal{\mathcal{H}^{\varphi}};\mathcal{\mathcal{E}^{\varphi}}; \mathcal{D}^{\varphi}\}\big), \big(\rrho_{\Phi}, W_{\Phi}, \{\mathcal{K}^{\Phi};\mathcal{F}^{\Phi}; \mathcal{O}^{\Phi}\}\big)\Big),\; \text{where}\;
     \end{equation*}
\begin{enumerate}
\item $\{\mathcal{\mathcal{H}^{\varphi}};\mathcal{\mathcal{E}^{\varphi}}; \mathcal{D}^{\varphi}\}$ and $\{\mathcal{K}^{\Phi};\mathcal{F}^{\Phi}; \mathcal{O}^{\Phi}\}$ are quantized domains in Hilbert spaces $\mathcal{H}^{\varphi}$ and $\mathcal{K}^{\Phi}$ respectively,
\item $\pi_{\varphi} \colon \mathcal{A} \to C^{\ast}_{\mathcal{E}^{\varphi}}(\mathcal{D}^{\varphi})$ is a unital local contractive $\ast$-homomorphism and the map $\rrho_{\Phi} \colon E \to C^{\ast}_{\mathcal{E}^{\varphi}, \mathcal{F}^{\Phi}}(\mathcal{D}^{\varphi}, \mathcal{O}^{\Phi})$ is a $\pi_{\varphi}$-morphism,
\item $V_{\varphi} \colon \mathcal{H} \to \mathcal{H}^{\varphi}$ and $W _{\Phi}\colon \mathcal{K}\to \mathcal{K}^{\Phi} $ are contractions,
\end{enumerate}
such that
\begin{equation*}
 V_{\varphi}(\mathcal{E}) \subseteq \mathcal{E}^{\varphi},\;  \; \varphi(a) \subseteq V_{\varphi}^{\ast}\pi_{\varphi}(a)V_{\varphi} ,\; \text{for all} \; a \in \mathcal{A}
\end{equation*}
 and
 \begin{equation*}
 W_{\Phi}(\mathcal{F}) \subseteq \mathcal{F}^{\Phi},\;  \;\Phi(x)\subseteq W_{\Phi}^{\ast}\rrho_{\Phi}(x)V_{\varphi}, \; \text{for all}\; x \in E.
 \end{equation*}
 \end{thm}

 \begin{proof} Since $\varphi \in \cpccloc$, by Theorem \ref{Theorem: DosievStinespring}, there is a Stinespring's representation for $\varphi$ denoted by the  triple $\big(\pi_{\varphi}, V_{\varphi}, \{\mathcal{H}^{\varphi};\mathcal{E}^{\varphi};
 \mathcal{D}^{\varphi}\}\big)$, where $\{\mathcal{H}^{\varphi};\mathcal{E}^{\varphi};
 \mathcal{D}^{\varphi}\}$ is a quantized domain in a Hilbert space $\mathcal{H}^{\varphi}$ with the upward filtered family $\mathcal{E}^{\varphi} = \{\mathcal{H}_{\ell}^{\varphi}:\; \ell \in \Omega\}$ and the union space $\mathcal{D}^{\varphi}$. Further, $V_{\varphi} \colon \mathcal{H} \to \mathcal{H}^{\varphi}$ is a contraction and $\pi_{\varphi} \colon \mathcal{A} \to
 C^{\ast}_{\mathcal{E}^{\varphi}}(\mathcal{D}^{\varphi})$ is a unital local contractive $\ast$-homomorphism such that
 	\begin{equation*}
 	V_{\varphi}(\mathcal{E}) \subseteq \mathcal{E}^{\varphi} \; \; \text{and}\; \; \varphi(a) \subseteq V_{\varphi}^{\ast}\pi_{\varphi}(a)V_{\varphi}, \; \text{for all}\; a \in \mathcal{A}.
 	\end{equation*}
With out loss of generality, by Proposition \ref{Proposition: minimality Dosiev}, we  consider  $\pi_{\varphi}$ to be minimal Stinespring representation for $\varphi$. That is, $\mathcal{H}_{\ell}^{\varphi} = [\pi_{\varphi}(\mathcal{A})V_{\varphi}\mathcal{H}_{\ell}]$, for every $\ell \in \Omega$ and consequently,
$ [\pi_{\varphi}(\mathcal{A})V_{\varphi}\mathcal{D}]= \mathcal{H}^{\varphi}.$
   Now we construct a quantized domain $\{\mathcal{K}^{\Phi};\mathcal{F}^{\Phi};\mathcal{O}^{\Phi}\}$ as follows:  Let us define
  \begin{equation*}
      \mathcal{K}^{\Phi}_{\ell}: = [\Phi(E)\mathcal{H}_{\ell}], \; \text{for each}\; \ell \in \Omega;
  \end{equation*}
  the union space $\mathcal{O}^{\Phi}:= \bigcup\limits_{\ell \in \Omega}\mathcal{K}^{\Phi}_{\ell}$ and the Hilbert space $\mathcal{K}^{\Phi}:=  \overline{\mathcal{O}^{\Phi}} $. Then the family denoted by $\mathcal{F}^{\Phi}:= \{\mathcal{K}_{\ell}^{\Phi}:\; \ell \in \Omega\}$ is an upward filtered family since
  \begin{equation*}
      \Phi(x)\xi \in \mathcal{H}_{\ell} \subseteq \mathcal{H}_{\ell^{\prime}}\; \text{for every}\; x\in E \; \text{and}\; \xi \in \mathcal{H}_{\ell}, \ell \leq \ell^{\prime}.
  \end{equation*}
  This  implies that
  $\{\mathcal{K}^{\Phi};\mathcal{F}^{\Phi};\mathcal{O}^{\Phi}\}$ is a quantized domain in the Hilbert space
  $\mathcal{K}^{\Phi}$.
  Note that $\mathcal{K}_{\ell}^{\Phi} \subseteq \mathcal{K_{\ell}}$ for every $\ell \in \Omega$ and so $\mathcal{K}^{\Phi} \subseteq \mathcal{K}$ by the construction.

 For each $x \in E$, we define $\rrho_{\Phi}(x) \colon \mathcal{D}^{\varphi} \to \mathcal{O}^{\Phi}$  by
 \begin{equation*}
     \rrho_{\Phi}(x)\Big(\sum\limits_{i=1}^{n}\pi_{\varphi}(a_{i})V_{\varphi}h_{i}\Big) = \sum\limits_{i=1}^{n}\Phi(xa_{i})h_{i},
 \end{equation*}
 for $i\in \{1,2,\cdots n\}, n\in \mathbb{N}$, $a_{i}\in \mathcal{A}$, $h_{i}\in \mathcal{H}_{\ell}$ for some $\ell \in \Omega$. First, we show that $\rrho_{\Phi}(x)$ is well defined, for all $x \in E$. For each $\ell \in \Omega$, we see that
 \begin{align*}
     \Big\| \rrho_{\Phi}(x)\big|_{\mathcal{H}_{\ell}^{\varphi}}\Big(\sum\limits_{i=1}^{n}\pi_{\varphi}(a_{i})V_{\varphi}h_{i}\Big)\Big\|^{2}_{\mathcal{K}^{\Phi}_{\ell}}
     &= \Big\| \Big(\sum\limits_{i=1}^{n}\Phi(xa_{i})h_{i}\Big)\Big\|^{2}_{\mathcal{K}^{\Phi}_{\ell}}\\
     &= \sum\limits_{i,j=1}^{n}\Big\langle \Phi(xa_{i})h_{i}, \Phi(xa_{j})h_{j} \Big\rangle_{\mathcal{K}^{\Phi}_{\ell}}\\
     &= \sum\limits_{i,j=1}^{n}\Big\langle h_{i}, \Phi(xa_{i})^{\ast}\Phi(xa_{j})h_{j} \Big\rangle_{\mathcal{H}_{\ell}}\\
     &= \sum\limits_{i,j=1}^{n}\Big\langle h_{i}, \varphi\big(a_{i}^{\ast}\langle x, x\rangle a_{j}\big)h_{j} \Big\rangle_{\mathcal{H}_{\ell}}\\
     &=\sum\limits_{i,j=1}^{n}\Big\langle h_{i}, V_{\varphi}^{\ast}\pi_{\varphi}\big(a_{i}^{\ast}\langle x, x\rangle a_{j}\big)V_{\varphi}h_{j} \Big\rangle_{\mathcal{H}_{\ell}}\\
     &=\sum\limits_{i,j=1}^{n}\Big\langle \pi_{\varphi}(a_{i})V_{\varphi} h_{i}, \pi_{\varphi}\big(\langle x, x\rangle \big)\pi_{\varphi}(a_{j})V_{\varphi}h_{j} \Big\rangle_{\mathcal{H}_{\ell}^{\varphi}}\\
     &=\Big\langle \sum\limits_{i=1}^{n} \pi_{\varphi}(a_{i})V_{\varphi} h_{i}, \pi_{\varphi}\big(\langle x, x\rangle \big) \sum\limits_{j=1}^{n} \pi_{\varphi}(a_{j})V_{\varphi}h_{j} \Big\rangle_{\mathcal{H}_{\ell}^{\varphi}}\\
     &\leq \Big\|\pi_{\varphi}\big(\langle x, x\rangle\big)|_{{\mathcal H}_{\ell }^{\varphi }}\Big\|\Big\| \sum\limits_{i=1}^{n} \pi_{\varphi}(a_{i})V_{\varphi} h_{i}\Big\|^{2}_{\mathcal{H}_{\ell}^{\varphi}}\\
     &\leq \|\langle x, x\rangle \|^{2}_{\alpha }\Big\| \sum\limits_{i=1}^{n} \pi_{\varphi}(a_{i})V_{\varphi} h_{i}\Big\|^{2}_{\mathcal{H}_{\ell}^{\varphi}},
 \end{align*}
for some $\alpha \in \Lambda$. This implies that $\rrho_{\Phi}(x)$ is well defined  and $\rrho_{\Phi}(x)\big|_{\mathcal{H}_{\ell}^{\varphi}} \in
 \mathcal{B}(\mathcal{H}_{\ell}^{\varphi}, \mathcal{K}_{\ell}^{\Phi})$, for all $x \in E$,  $\ell \in \Omega$. Equivalently,
 $\rrho_{\Phi}(x) \in C^{\ast}_{\mathcal{E}^{\varphi},\mathcal{F}^{\Phi}}(\mathcal{D}^{\varphi}, \mathcal{O}^{\Phi})$, for all
 $x \in E$. It is enough to show that $\rrho_{\Phi}$ is a $\pi_{\varphi}$-morphism. Let $x,y \in E$, $a_{i}, b_{j} \in \mathcal{A}$, $h_{i},
 g_{j}\in \mathcal{H}_{\ell}$, for $i\in \{1,2,3,\cdots n\}$, $j \in\{1,2,3,\cdots m\}$, $m,n \in \mathbb{N}$. Then
computations similar to the computations above yield,
\begin{align*}  & & \Big \langle \sum\limits_{i=1}^{n}\pi_{\varphi}(a_{i})V_{\varphi}h_{i}, \rrho_{\Phi}(x)^{\ast}\rrho_{\Phi}(y) \sum\limits_{j=1}^{n}\pi_{\varphi}(b_{j})V_{\varphi}g_{j} \Big\rangle_{\mathcal{K}_{\ell}^{\Phi}}\\
& &=
\Big\langle \sum\limits_{i=1}^{n}\pi_{\varphi}(a_{i})V_{\varphi}h_{i}, \; \pi_{\varphi}\big( \langle x,y\rangle\big)\sum\limits_{j=1}^{m}\pi_{\varphi}(b_{j})V_{\varphi}g_{j} \Big\rangle_{\mathcal{H}_{\ell}^{\varphi}}.
\end{align*}
%
 It follows that $\rrho_{\Phi}(x)^{\ast}\rrho_{\Phi}(y) = \pi_{\varphi}(\langle x,y \rangle)$ on the dense set span$\{\pi_{\varphi}(\mathcal{A})V_{\varphi}\mathcal{H}_{\ell}\}$ of $\mathcal{H}_{\ell}^{\varphi}$ for every $\ell \in \Omega$. Since $\mathcal{D} = \bigcup\limits_{\ell \in \Omega}\mathcal{H}_{\ell}^{\varphi}$, we conclude that $\rrho_{\Phi}$ is a $\pi_{\varphi}$-morphism. We know by the construction that $\mathcal{K}^{\Phi}$ is a closed subspace of $\mathcal{K}$. Let $W_{\Phi}$ be the orthogonal projection of $\mathcal{K}$ onto $\mathcal{K}^{\Phi}$ thought of as an operator from $\mathcal{K}$ to $\mathcal{K}^{\Phi}$. Then $W_{\Phi}W_{\Phi}^{\ast} = I_{\mathcal{K}^{\Phi}}$, the identity operator on $\mathcal{K}^{\Phi}$. Moreover, if  $x\in E$ and $h \in \mathcal{D}$, then $h \in \mathcal{H}_\ell$ for some $\ell \in \Omega$ and
 \begin{equation*}
     W_{\Phi}^{\ast}\rrho_{\Phi}(x)V_{\varphi}h = \rrho_{\Phi}(x)\Big(\pi_{\varphi}(1)V_{\varphi}h\Big)= \Phi(x)h.
 \end{equation*}
 This shows that $\Phi(x) \subseteq W_{\Phi}^{\ast}\rrho_{\Phi}(x)V_{\varphi}$, for all $x \in E$.
 \end{proof}
\begin{note}
Let $\Phi \in C_{loc}(E, C^{*}_{\mathcal{E, F}}(\mathcal{D, O}))$ with the associated $\varphi \in \cpccloc$. We say a  pair of triples $\Big( \big( \pi_{\varphi}, V_{\varphi}, \{\mathcal{H}^{\varphi};\mathcal{E}^{\varphi};\mathcal{D}^{\varphi}\}\big),\; \big( \rrho_{\Phi}, W_{\Phi}, \{\mathcal{K}^{\Phi};\mathcal{F}^{\Phi};\mathcal{O}^{\Phi}\}\big)\Big)$ is a \emph{Stinespring representation} for the pair $(\varphi, \Phi)$ if it satisfies conditions $(1), (2)$ and $(3)$ of Theorem \ref{Theorem: Stinespringforphimaps}.
\end{note}


\begin{defn} \label{Definition: minimalityforphimaps}We call that a pair of triples $$\Big( \big( \pi_{\varphi}, V_{\varphi}, \{\mathcal{H}^{\varphi};\mathcal{E}^{\varphi};\mathcal{D}^{\varphi}\}\big),\; \big( \rrho_{\Phi}, W_{\Phi}, \{\mathcal{K}^{\Phi};\mathcal{F}^{\Phi};\mathcal{O}^{\Phi}\}\big)\Big)$$ is  \emph{minimial} Stinespring representaton for the pair $(\varphi, \Phi)$ if the following conditions hold:
\begin{enumerate}
    \item $\mathcal{E}^{\varphi} \subseteq \mathcal{E}$ and $\mathcal{H}^{\varphi}_{\ell} =[\pi_{\varphi}(\mathcal{A})V_{\varphi}\mathcal{H} _{\ell }]$ for every $\ell \in \Omega ;$
    \item $\mathcal{F}^{\Phi} \subseteq \mathcal{F}$ and $\mathcal{K}^{\Phi }_{\ell }=[
    \rho _{\Phi }(\mathcal {A})V_{\varphi }\mathcal{H} _{\ell}]$ for every $\ell \in \Omega .$
\end{enumerate}
\end{defn}

 Now we show that given any Stinespring representation we can modify the spaces so as to get a minimal Stinespring representation.
 \begin{prop}
 Let $\Big( \big( \pi_{\varphi}, V_{\varphi}, \{\mathcal{H}^{\varphi};\mathcal{E}^{\varphi};\mathcal{D}^{\varphi}\}\big),\; \big( \rrho_{\Phi}, W_{\Phi}, \{\mathcal{K}^{\Phi};\mathcal{F}^{\Phi};\mathcal{O}^{\Phi}\}\big)\Big) $ be a Stinespring representation for the pair $(\varphi, \Phi)$. Then there is another Stinespring representation
 \begin{equation*}
 \Big( \big( \widetilde{\pi}_{\varphi}, \widetilde{V}_{\varphi}, \{\widetilde{\mathcal{H}}^{\varphi};\widetilde{\mathcal{E}}^{\varphi};\widetilde{\mathcal{D}}^{\varphi}\}\big),\; \big( \widetilde{\rrho}_{\Phi}, \widetilde{W}_{\Phi}, \{\widetilde{\mathcal{K}}^{\Phi};\widetilde{\mathcal{F}}^{\Phi};\widetilde{\mathcal{O}}^{\Phi}\}\big)\Big)
 \end{equation*}
 such that
\begin{enumerate}
\item $\widetilde{\mathcal{E}}^{\varphi} \subseteq \mathcal{E}^{\varphi}$ and $\widetilde{\mathcal{F}}^{\Phi} \subseteq \mathcal{F}^{\Phi}$;
\item $ \widetilde{\mathcal{H}}^{\varphi} = [\widetilde{\pi}_{\varphi}(\mathcal{A})\widetilde{V}_{\varphi}{\mathcal{D}}] \;\; \text{and}\;\; \widetilde{\mathcal{K}}^{\Phi} = [\widetilde{\rrho}_{\Phi}(E)\widetilde{V}_{\varphi}{\mathcal{D}}]$.
\end{enumerate}
\end{prop}
\begin{proof}
Consider $\big( \pi_{\varphi}, V_{\varphi}, \{\mathcal{H}^{\varphi};\mathcal{E}^{\varphi};\mathcal{D}^{\varphi}\}\big)$ is a Stinespring representation for $\varphi$. By Proposition \ref{Proposition: minimality Dosiev}, there exists another Stinespring representation  $\big( \widetilde{\pi}_{\varphi}, \widetilde{V}_{\varphi}, \{\widetilde{\mathcal{H}}^{\varphi};\widetilde{\mathcal{E}}^{\varphi};\widetilde{\mathcal{D}}^{\varphi}\} \big)$ such that
 $   \widetilde{\mathcal{H}}^{\varphi} = [\widetilde{\pi}_{\varphi}(\mathcal{A})\widetilde{V}_{\varphi}\mathcal{D}].$
We recall from the construction of this representation that $\widetilde{\mathcal{H}}_{\ell}^{\varphi} = [{\pi}_{\varphi}(\mathcal{A}){V}_{\varphi}\mathcal{H}_{\ell}]$ for every $\ell \in \Omega$, $\widetilde{\mathcal{E}}^{\varphi}= \{\widetilde{\mathcal{H}}^{\varphi}_{\ell}:\; \ell \in \Omega\}$ is an upward filtered family with the union space $\widetilde{\mathcal{D}}^{\varphi} = \bigcup\limits_{\ell \in \Omega}\widetilde{\mathcal{H}}_{\ell}^{\varphi} $ is dense in $\widetilde{\mathcal{H}}^{\varphi}$. Clearly, $\widetilde{\mathcal{E}}^{\varphi} \subseteq \mathcal{E}^{\varphi}$. Also $\widetilde{V}_{\varphi} = V_{\varphi}$ and $\widetilde{\pi}_{\varphi}(a) = \pi_{\varphi}(a)\big|_{\widetilde{\mathcal{D}}^{\varphi}}$, for all $a \in \mathcal{A}$.

Similarly, let us define $\widetilde{\mathcal{K}}_{\ell}^{\Phi}:= [\rrho_{\Phi}(E)V_{\varphi}\mathcal{H}_{\ell}]$, for every $\ell \in \Omega$ and $\widetilde{\mathcal{O}}^{\Phi}:= \bigcup\limits_{\ell \in \Omega}\widetilde{\mathcal{K}}^{\Phi}_{\ell}$. Then $\widetilde{\mathcal{K}}^{\Phi}_{\ell} \subseteq \mathcal{K}^{\Phi}_{\ell}$ and ${\widetilde{K}}^{\Phi}_{\ell} \subseteq {\widetilde{K}}^{\Phi}_{\ell^{\prime}}$, for each $\ell \leq \ell^{\prime}$. It implies that $\widetilde{\mathcal{F}}^{\Phi} \subseteq {\mathcal{F}}^{\Phi}$ and $\widetilde{\mathcal{F}}^{\Phi}:= \{\widetilde{\mathcal{K}}^{\Phi}_{\ell}:\; \ell \in \Omega\}$ is an upward filtered family in the Hilbert space $\widetilde{\mathcal{K}}^{\Phi}:= \overline{\widetilde{\mathcal{O}}^{\Phi}}$. As a result, $\big\{\widetilde{\mathcal{K}}^{\Phi};\widetilde{\mathcal{F}}^{\Phi};\widetilde{\mathcal{O}}^{\Phi}\big\}$ is a quantized domain in $\widetilde{\mathcal{K}}^{\Phi}$.  Define $\widetilde{\rrho}_{\Phi}(x)\colon \widetilde{\mathcal{D}}^{\varphi} \to \widetilde{\mathcal{O}}^{\Phi}$ by
\begin{equation*}
    \widetilde{\rrho}_{\Phi}(x):= \rrho_{\Phi}(x)\big|_{\widetilde{\mathcal{D}}^{\varphi}},\; \text{for all}\; x \in E.
\end{equation*}
Clearly, $\widetilde{\rrho}_{\Phi}(x)\big|_{\widetilde{\mathcal{H}}_{\ell}^{\varphi}} \in \mathcal{B}\big(\widetilde{\mathcal{H}}_{\ell}^{\varphi}, \widetilde{\mathcal{K}}_{\ell}^{\Phi}\big)$ for each $\ell \in \Omega$. This implies that the following map
$$\widetilde{\rrho}_{\Phi}\colon E \to C^{\ast}_{\widetilde{\mathcal{E}}^{\varphi}, \widetilde{\mathcal{F}}^{\Phi}}(\widetilde{\mathcal{D}}^{\varphi}, \widetilde{\mathcal{O}}^{\Phi})$$ is well defined. Also $\widetilde{\rrho}_{\Phi}$ is a $\widetilde{\pi}_{\varphi}$-morphism since
\begin{align*}
    \Big\langle \pi_{\varphi}(a)V_{\varphi}h,\; \widetilde{\rrho}_{\Phi}(x)^{\ast}\;\widetilde{\rrho}_{\Phi}(y)\big(\pi_{\varphi}(b)V_{\varphi}g\big)\Big\rangle_{\widetilde{\mathcal{K}}_{\ell}^{\Phi}}
    &= \Big\langle \widetilde{\rrho}_{\Phi}(x) \pi_{\varphi}(a)V_{\varphi}h,\;\widetilde{\rrho}_{\Phi}(y)\big(\pi_{\varphi}(b)V_{\varphi}g\big)\Big\rangle_{\widetilde{\mathcal{K}}_{\ell}^{\Phi}}\\
    &= \Big\langle \Phi(xa)h,\;\Phi(yb)h \Big\rangle_{{\mathcal{K}}_{\ell}} \\
    &= \Big\langle h,\;\Phi(xa)^{\ast}\Phi(yb)g \Big\rangle_{{\mathcal{H}}_{\ell}}\\
    &= \Big\langle h,\;{\varphi}(a^{\ast}\langle x, y\rangle b)g \Big\rangle_{{\mathcal{H}}_{\ell}}\\
    &=\Big\langle \pi_{\varphi}(a)V_{\varphi}h,\;\pi_{\varphi}\big(\langle x, y\rangle \big) \pi_{\varphi}(b)V_{\varphi}g \Big\rangle_{\widetilde{\mathcal{H}}_{\ell}^{\varphi}},
\end{align*}
for every $a,b \in \mathcal{A}$, $g,h  \in \mathcal{H}_{\ell}, \ell \in \Omega$ and $x,y \in E$. Let $\widetilde{W}_{\Phi}$ be the orthogonal projection of $\mathcal{K}$ onto $\widetilde{\mathcal{K}}^{\Phi}$. If $h \in \mathcal{D}$, then $h \in \mathcal{H}_{\ell}$ for some $\ell \in \Omega$ and
\begin{equation*}
    \widetilde{W}^{\ast}_{\Phi}\widetilde{\rrho}_{\Phi}(x)\widetilde{V}_{\varphi}h = \widetilde{\rrho}_{\Phi}(x)\widetilde{\pi}_{\varphi}(1){V}_{\varphi}h = \Phi(x)h.
\end{equation*}
Thus $\Phi(x) \subseteq \widetilde{W}^{\ast}_{\Phi}\widetilde{\rrho}_{\Phi}(x)\widetilde{V}_{\varphi}$, for all $x \in E$. Therefore, the pair of triples $$\Big( \big( \widetilde{\pi}_{\varphi}, \widetilde{V}_{\varphi}, \{\widetilde{\mathcal{H}}^{\varphi};\widetilde{\mathcal{E}}^{\varphi};\widetilde{\mathcal{D}}^{\varphi}\}\big),\; \big( \widetilde{\rrho}_{\Phi}, \widetilde{W}_{\Phi}, \{\widetilde{\mathcal{K}}^{\Phi};\widetilde{\mathcal{F}}^{\Phi};\widetilde{\mathcal{O}}^{\Phi}\}\big)\Big)$$ is the desired Stinespring representation for $(\varphi, \Phi)$.
\end{proof}

 Now, we show that any two minimal Stinespring representations for the pair $(\varphi, \Phi)$ are unitary equivalant.
 \begin{thm}\label{Theorem: uniqunessforphimaps} Suppose that $\Big(\big(\widetilde{\pi}_{\varphi}, \widetilde{V}_{\varphi}, \{\widetilde{\mathcal{H}}^{\varphi}; \widetilde{\mathcal{E}}^{\Phi}; \widetilde{D}^{\varphi}\}\big),\; \big(\widetilde{\rrho}_{\Phi}, \widetilde{W}_{\Phi}, \{\widetilde{\mathcal{K}}^{\Phi};\widetilde{\mathcal{F}}^{\Phi};\widetilde{\mathcal{O}}^{\Phi}\}\big) \Big)$ \; and \\ $\Big(\big(\widehat{\pi}_{\varphi}, \widehat{V}_{\varphi}, \{\widehat{\mathcal{H}}^{\varphi}; \widehat{\mathcal{E}}^{\Phi}; \widehat{D}^{\varphi}\}\big),\; \big(\widehat{\rrho}_{\Phi}, \widehat{W}_{\Phi}, \{\widehat{\mathcal{K}}^{\Phi};\widehat{\mathcal{F}}^{\Phi};\widehat{\mathcal{O}}^{\Phi}\}\big) \Big)$ are  two minimal Stinespring representations   for the pair $(\varphi, \Phi)$. Then there exist unitary operators $U_{\varphi} \colon \widetilde{\mathcal{H}}^{\varphi} \to \widehat{\mathcal{H}}^{\varphi}$ and $U_{\Phi}\colon \widetilde{\mathcal{K}}^{\Phi} \to \widehat{\mathcal{K}}^{\Phi}$ such that
 \begin{enumerate}
     \item $U_{\varphi}\widetilde{V}_{\varphi} = \widehat{V}_{\varphi}$, $U_{\varphi}\widetilde{\pi}_{\varphi}(a) = \widehat{\pi}_{\varphi}(a)U_{\varphi}$, for all $a \in \mathcal{A}$ and
     \item $U_{\Phi}\widetilde{W}_{\Phi} = \widehat{W}_{\Phi}$, $U_{\Phi} \widetilde{\rrho}_{\Phi}(x) = \widehat{\rrho}_{\Phi}(x)U_{\Phi}$, for all $x \in E$.
 \end{enumerate}
 Equivalently, for every $a \in \mathcal{A}, x \in E$, the following diagram commutes:
 \begin{center}
     \begin{tikzcd}
      & & &\widetilde{\mathcal{H}}^{\varphi}\ar[dddd, "U_{\varphi}", bend right=40, blue] & \widetilde{\mathcal{H}}^{\varphi} & \widetilde{\mathcal{K}}^{\Phi}\ar[dddd, "U_{\Phi}", bend left=40, magenta] & & & \\
      & & & \widetilde{\mathcal{D}}^{\varphi} \ar[dd, "U_{\varphi}|_{\widetilde{\mathcal{D}}^{\varphi}}", blue]\ar[u, hook, blue]\ar[r, "\widetilde{\pi}_{\varphi}(a)", blue] & \widetilde{\mathcal{D}}^{\varphi}\ar[dd, "U_{\varphi}|_{\widetilde{\mathcal{D}}^{\varphi}}", blue]\ar[u, hook]\ar[r, "\widetilde{\rrho}_{\Phi}(x)", magenta] & \widetilde{\mathcal{O}}^{\Phi}\ar[dd, "U_{\Phi}|_{\widetilde{\mathcal{O}}^{\Phi}}", magenta]\ar[u, hook, magenta] & & &\\
      \mathcal{H} \ar[uurrr, "\widetilde{V}_{\varphi}"] \ar[ddrrr, "\widehat{V}_{\varphi}"] & & &  & & &  & & \mathcal{K}\ar[uulll, "\widetilde{W}_{\Phi}"] \ar[ddlll, "\widehat{W}_{\Phi}"]\\
     & & &  \widehat{\mathcal{D}}^{\varphi}\ar[d, hook, blue] \ar[r, "\widehat{\pi}_{\varphi}(a)", blue] & \widehat{\mathcal{D}}^{\varphi} \ar[d, hook]\ar[r, "\widehat{\rrho}_{\Phi}(x)", magenta]& \widehat{\mathcal{O}}^{\Phi} \ar[d, hook, magenta]& & & \\
     & & & \widehat{\mathcal{H}}^{\varphi} & \widehat{\mathcal{H}}^{\varphi} & \widehat{\mathcal{K}}^{\Phi} & & &
     \end{tikzcd}
 \end{center}
 \end{thm}
 \begin{proof} Since $\widetilde{\mathcal{E}}^{\varphi} \subseteq \mathcal{E}$, $\widetilde{\mathcal{H}}^{\varphi} = [\widetilde{\pi}_{\varphi}(\mathcal{A})\widetilde{V}_{\varphi}\mathcal{D}]$ and $\widehat{\mathcal{E}}^{\varphi} \subseteq \mathcal{E}$, $\widehat{\mathcal{H}}^{\varphi} = [\widehat{\pi}_{\varphi}(\mathcal{A})\widehat{V}_{\varphi}\mathcal{D}]$, we see that $\big(\widetilde{\pi}_{\varphi}, \widetilde{V}_{\varphi}, \{\widetilde{\mathcal{H}}^{\varphi}; \widetilde{\mathcal{E}}^{\Phi}; \widetilde{D}^{\varphi}\}\big)$ and $\big(\widehat{\pi}_{\varphi}, \widehat{V}_{\varphi}, \{\widehat{\mathcal{H}}^{\varphi}; \widehat{\mathcal{E}}^{\Phi}; \widehat{D}^{\varphi}\}\big)$  are two minimal Stinespring representations for $\varphi$. It is clear from Theorem \ref{Theorem: UniquenessforDosiev} that there exists a unitary operator $U_{\varphi} \colon \widetilde{\mathcal{H}}^{\varphi} \to \widehat{\mathcal{H}}^{\varphi}$ such that $$U_{\varphi}\widetilde{V}_{\varphi} = \widehat{V}_{\varphi},\;\; U_{\varphi}\widetilde{\pi}_{\varphi}(a) = \widehat{\pi}_{\varphi}(a)U_{\varphi}, \; \text{for all}\; a \in \mathcal{A}.$$
 Now we define $U_{\Phi} \colon \text{span}\big\{\widetilde{\rrho}_{\Phi}(E)\widetilde{V}_{\varphi}\mathcal{D}\big\} \to \text{span}\big\{\widehat{\rrho}_{\Phi}(E)\widehat{V}_{\varphi}\mathcal{D}\big\}$ by
 \begin{equation*}
     U_{\Phi} \Big(\sum\limits_{i=1}^{n}\widetilde{\rrho}_{\Phi}(x_{i})\widetilde{V}_{\varphi}h_{i} \Big) = \sum\limits_{i=1}^{n}\widehat{\rrho}_{\Phi}(x_{i})\widehat{V}_{\varphi}h_{i},
 \end{equation*}
 for all $ x_{i} \in E, h_{i}\in \mathcal{D}, i \in \{1,2,3,\cdots n\}, n \in \mathbb{N}$. Then
 \begin{align*}
     \Big\|U_{\Phi} \Big(\sum\limits_{i=1}^{n}\widetilde{\rrho}_{\Phi}(x_{i})\widetilde{V}_{\varphi}h_{i} \Big)\Big\|^{2}_{\widehat{\mathcal{K}}^{\Phi}_{\ell}} &= \Big\|\sum\limits_{i=1}^{n}\widehat{\rrho}_{\Phi}(x_{i})\widehat{V}_{\varphi}h_{i}\Big\|^{2}_{\widehat{\mathcal{K}}^{\Phi}_{\ell}}\\
     &= \sum\limits_{i,j = 1}^{n} \Big\langle h_{i}, \widehat{V}_{\varphi}^{\ast}\widehat{\rrho}_{\Phi}(x_{i})^{\ast}\widehat{\rrho}_{\Phi}(x_{j})\widehat{V}_{\varphi}h_{j}\Big\rangle_{\mathcal{H}_{\ell}}\\
     &=\sum\limits_{i,j = 1}^{n} \Big\langle h_{i}, \widehat{V}_{\varphi}^{\ast}\widehat{\pi}_{\varphi}\big(\langle x_{i} , x_{j}\rangle\big)\widehat{V}_{\varphi}h_{j}\Big\rangle_{\mathcal{H}_{\ell}}\\
     &=\sum\limits_{i,j = 1}^{n} \Big\langle h_{i}, \widetilde{V}_{\varphi}^{\ast}\widetilde{\pi}_{\varphi}\big(\langle x_{i} , x_{j}\rangle\big)\widetilde{V}_{\varphi}h_{j}\Big\rangle_{\mathcal{H}_{\ell}}\\
     &=\Big\|\Big(\sum\limits_{i=1}^{n}\widetilde{\rrho}_{\Phi}(x_{i})\widetilde{V}_{\varphi}h_{i} \Big)\Big\|^{2}_{\widetilde{\mathcal{K}}^{\Phi}_{\ell}},
 \end{align*}
 for every $\ell \in \Omega$. This implies that $U_{\Phi}$ is well defined and is an isometry. Thus it can be extended to the whole of $\widetilde{\mathcal{K}}^{\Phi}$. Denote this extension by $U_{\Phi}$ it self. Note that $U_{\Phi}$ is a unitary since it is an isometry onto $\widehat{\mathcal{K}}^{\Phi}$. We know from hypothesis that
 \begin{equation}\label{Equation: above}
     \Phi(x)h = \widetilde{W}_{\Phi}^{\ast}\widetilde{\rrho}_{\Phi}(x)\widetilde{V}_{\varphi}h = \widehat{W}_{\Phi}^{\ast}\widehat{\rrho}_{\Phi}(x)\widehat{V}_{\varphi}h, \; \text{for every}\; h \in \mathcal{D}, x \in E.
 \end{equation}
   By the definition of $U_{\Phi}$, we can rewrite the above Equation (\ref{Equation: above}) as,
 \begin{equation*}
     \Big(\widetilde{W}_{\Phi}^{\ast} - \widehat{W}_{\Phi}^{\ast} U_{\Phi}\Big) \widetilde{\rrho}_{\Phi}(x)\widetilde{V}_{\varphi}h = 0, \; \text{for all}\; h \in \mathcal{D},  x \in E.
 \end{equation*}
 Since $\widetilde{\mathcal{H}}^{\varphi}= [\widetilde{\rrho}_{\Phi}(E)\widetilde{V}_{\varphi}\mathcal{D}]$, we have $\widetilde{W}_{\Phi}^{\ast} = \widehat{W}_{\Phi}^{\ast} U_{\Phi}$. By taking adjoint and multiplying with $U_{\Phi}$ on both sides, we get that $U_{\Phi}\widetilde{W}_{\Phi} = \widehat{W}_{\Phi}$.

 Now we show that $U_{\Phi} \widetilde{\rrho}_{\Phi}(x) =\widehat{\rrho}_{\Phi}(x)U_{\Phi}$, for all $x \in E$.
 Using the fact that $U_{\varphi}$ is a unitary and $U_{\varphi}\big(\widetilde{\mathcal{D}}^{\varphi}\big)=\widehat{\mathcal{D}}^{\varphi},$ it is easy to see that $\text{dom}\big(U_{\Phi} \widetilde{\rrho}_{\Phi}(x)\big)=\mathcal{D}^{\varphi}=\text{dom}\big( \widehat{\rrho}_{\Phi}(x)U_{\Phi} \big)$,
 for all $x \in E$. Let $x_{i} \in E, a_{j}\in \mathcal{A}$, $h_{i}, g_{j}\in {\mathcal{H}}_{\ell}$ for $i \in \{1,2,3, \cdots n\}$, $j\in
 \{1,2,3,\cdots m\}$, $n,m \in \mathbb{N}$. Then
 \begin{align*}
     \Big\langle \sum\limits_{i=1}^{n} \widehat{\rrho}_{\Phi}(x_{i})\widehat{V}_{\varphi}h_{i},\; & U_{\Phi}\widetilde{\rrho}_{\Phi}(x)\Big(\sum\limits_{j=1}^{m} \widetilde{\pi}_{\varphi}(a_{j})\widetilde{V}_{\varphi}g_{j}\Big)\Big\rangle_{\widehat{\mathcal{K}}_{\ell}^{\Phi}} \\
     &= \sum\limits_{i,j=1}^{n,m}\Big\langle  h_{i},\; \widehat{V}_{\varphi}^{\ast}\widehat{\rrho}_{\Phi}(x_{i})^{\ast}\widehat{\rrho}_{\Phi}(xa_{j})\widehat{V}_{\varphi}g_{j}\Big\rangle_{\mathcal{H}_{\ell}}\\
     &= \sum\limits_{i,j=1}^{n,m}\Big\langle  h_{i},\; \widehat{V}_{\varphi}^{\ast}\widehat{\pi}_{\varphi}\big(\langle x_{i}, x\rangle a_{j}\big)\widehat{V}_{\varphi}g_{j}\Big\rangle_{\mathcal{H}_{\ell}}\\
     &= \sum\limits_{i,j=1}^{n,m}\Big\langle  h_{i},\; \widehat{V}_{\varphi}^{\ast}\widehat{\pi}_{\varphi}\big(\langle x_{i}, x\rangle \big)\widehat{\pi}_{\varphi}(a_{j})\widehat{V}_{\varphi}g_{j}\Big\rangle_{\mathcal{H}_{\ell}}\\
     &= \sum\limits_{i,j=1}^{n,m}\Big\langle  h_{i},\; \widetilde{V}_{\varphi}^{\ast}\widehat{\rrho}_{\Phi}(x_{i})^{\ast}\widehat{\rrho}_{\Phi}(x)U_{\Phi}\Big(\widetilde{\pi}_{\varphi}(a_{j})\widetilde{V}_{\varphi}g_{j}\Big)\Big\rangle_{\mathcal{H}_{\ell}}\\
     &= \Big\langle \sum\limits_{i=1}^{n} \widehat{\rrho}_{\Phi}(x_{i})\widehat{V}_{\varphi}h_{i},\; \widehat{\rrho}_{\Phi}(x)U_{\Phi}\Big(\sum\limits_{j=1}^{m}\widetilde{\pi}_{\varphi}(a_{j})\widetilde{V}_{\varphi}g_{j}\Big)\Big\rangle_{\widehat{\mathcal{K}}_{\ell}^{\Phi}},
 \end{align*}
 Thus $U_{\Phi} \widetilde{\rrho}_{\Phi}(x) = \widehat{\rrho}_{\Phi}(x)U_{\Phi}$ on the dense set $\text{span}\{\widetilde{\pi}_{\varphi}(\mathcal{A})\widetilde{V}_{\varphi}\mathcal{D}\}$ and hence they are equal on $\widetilde{\mathcal{D}}^{\varphi}$. We conclude that $U_{\Phi} \widetilde{\rrho}_{\Phi}(x) = \widehat{\rrho}_{\Phi}(x)U_{\Phi}$, for all $x \in E$.
 \end{proof}
\section{Radon-Nikodym theorem for Local CP-inducing maps}

In this section, motivated by the work of Joita \cite{Joita-Main}, we provide order relation on the class of all local CP-inducing maps defined on some Hilbert module over a locally $C^{\ast}$-algebra. Throughout this section $E$ denotes a Hilbert module over the locally $C^{\ast}$-algebra $\mathcal{A}$ and $\{\mathcal{H}; \mathcal{E}; \mathcal{D}\}$, $\{\mathcal{K}; \mathcal{F}; \mathcal{O}\}$ are quantized domains.

Let us begin our discussion with the following definition which describes order on the class $C_{loc}(E, C^{*}_{\mathcal{E, F}}(\mathcal{D, O}))$. Also, for any given $\Phi \in C_{loc}(E, C^{*}_{\mathcal{E, F}}(\mathcal{D, O}))$, we compute Randon-Nikodym derivative and characterize all the maps that are dominated by $\Phi$.

\begin{defn}\label{Definition: orderrelation}Let $\Phi, \Psi\in C_{loc}(E, C^{*}_{\mathcal{E, F}}(\mathcal{D, O}))$, where $\Phi $ is a $\varphi $-map and $\Psi $ is a $\psi $-map, $\varphi , \psi \in   \cpccloc  .$ We say that $\Psi$ is \emph{dominated} by $\Phi$ with the notation $\Psi \preceq \Phi$, if
$$\varphi-\psi \in \cpccloc.$$
\end{defn}
\begin{rmk} Let $\Phi_{i}\in \mathcal{C}_{loc}(E, C^{*}_{\mathcal{E, F}}(\mathcal{D, O}))$ for $i=1,2,3.$ Then following are some immediate observations:
\begin{enumerate}
     \item [(i)] $\Phi_{1} \preceq \Phi_{1}.$
    \item [(ii)] If $\Phi_{1} \preceq \Phi_{2}$ and $\Phi_{2} \preceq \Phi_{3}$, then $\Phi_{1} \preceq \Phi_{3}.$
    \item [(ii)] If $\Phi_{1} \preceq \Phi_{2}$ and $\Phi_{2} \preceq \Phi_{1}$, then $\varphi_{1}=\varphi_{2}.$
    \end{enumerate}
\end{rmk}
\begin{defn}\label{Definition: equivalent}
Let $\Phi_{1},\Phi_{2}\in \mathcal{C}_{loc}(E, C^{*}_{\mathcal{E, F}}(\mathcal{D, O})),$ where $\Phi_{i}$ is $\varphi_{i}$-map, $\varphi_{i} \in \cpccloc$ for $i=1,2$. We say that $\Phi_{1}$ is equivalent to $\Phi_{2}$, denoted by $\Phi_{1}\sim \Phi_{2}$ if and only if $\varphi_{1}= \varphi_{2}$.
\end{defn}
Note that $\sim$ is an equivalence relation on $\cloc$ and the equivalence class of $\Phi \in \cloc$ is denoted by $[\Phi]$. The following theorem characterizes the equivalence class via partial isometries.
\begin{thm}
Let $\Phi, \Psi \in \mathcal{C}_{loc}(E, C^{*}_{\mathcal{E, F}}(\mathcal{D, O}))$. Then $\Psi \in [\Phi]$ if and only if there is a partial isometry $W \in \mathcal{B}(\mathcal{K})$ satisfying, $W^{\ast}W = P_{\mathcal{K}^{\Psi}}$ and $WW^{\ast} = P_{\mathcal{K}^{\Phi}}$ such that,
\begin{equation*}
    \Psi(x) = W\Phi(x), \; \text{for all}\; x \in E.
\end{equation*}
\begin{proof}
Let $\Phi$ and $\Psi$ be $\varphi$-map and $\psi$-map respectively, where $\varphi \text{~and~} \psi \in \cpccloc$.
With out loss of generality, we assume that $\Big( (\pi_{\varphi}, V_{\varphi}, \{\mathcal{H}^{\varphi}; \mathcal{E}^{\varphi};
\mathcal{D}^{\varphi}\}),\; (\rrho_{\Phi},  W_{\Phi},  \{\mathcal{K}^{\Phi}; \mathcal{F}^{\Phi}; \mathcal{O}^{\Phi}\})\Big)$ and
$\Big( (\pi_{\psi}, V_{\psi}, \{\mathcal{H}^{\psi}; \mathcal{E}^{\psi}; \mathcal{D}^{\psi}\}),\; (\rrho_{\Psi}, W_{\Psi}, \{\mathcal{K}^{\Psi};
\mathcal{F}^{\Psi}; \mathcal{O}^{\Psi}\})\Big)$ are the minimal Stinespring's triples associated to $(\varphi, \Phi)$ and $(\psi, \Psi)$
respectively. Suppose that $\Psi \in [\Phi]$, then $\varphi = \psi$. Now we define
\begin{equation*}
    U\big(\sum\limits_{i=1}^{n}\rrho_{\Phi}(x_{i})V_{\varphi}h_{i}\big) = \sum\limits_{i=1}^{n}\rrho_{\Psi}(x_{i})V_{\psi}h_{i},
\end{equation*}
for every $x_{i}\in E, h_{i} \in \mathcal{H}_{\ell}$, $i \in \{1, 2, 3,\hdots, n\}$ and $\ell \in \Omega$. Then
\begin{align*}
    \Big\|\sum\limits_{i=1}^{n}\rrho_{\Psi}(x_{i})V_{\psi}h_{i}\Big\|^{2}_{\mathcal{K}_{\ell}^{\Psi}} &= \sum\limits_{i,j=1}^{n}\big\langle h_{i},\; V_{\psi}^{\ast}\rrho_{\Psi}(x_{i})^{\ast}\rrho_{\Psi}(x_{j})V_{\psi}h_{j}\big\rangle_{\mathcal{H}_{\ell}}\\
    &= \sum\limits_{i,j=1}^{n}\big\langle h_{i},\; V_{\varphi}^{\ast}\pi_{\varphi}(\langle x_{i}, x_{j}\rangle)V_{\varphi}h_{j}\big\rangle_{\mathcal{H}_{\ell}}\\
    &= \sum\limits_{i,j=1}^{n}\big\langle h_{i},\; V_{\varphi}^{\ast}\rrho_{\Phi}(x_{i})^{\ast}\rrho_{\Phi}(x_{j})V_{\varphi}h_{j}\big\rangle_{\mathcal{H}_{\ell}}\\
    &=\Big\|\sum\limits_{i=1}^{n}\rrho_{\Phi}(x_{i})V_{\varphi}h_{i}\Big\|^{2}_{\mathcal{K}_{\ell}^{\Phi}}.
\end{align*}
It shows that $U$ is an isometry that maps a dense subset of $\mathcal{K}_{\ell}^{\Phi}$ to a dense subset of $\mathcal{K}_{\ell}^{\Psi}$, for every $\ell \in \Omega$. Therefore it can be extended to a unitary, again denote it by $U \in \mathcal{B}(\mathcal{K}^{\Phi}, \mathcal{K}^{\Psi})$.

If we take $W = W_{\Psi}^{\ast}UW_{\Phi}$, then clearly, we see that $W^{\ast}W = P_{\mathcal{K}^{\Phi}}$ and $WW^{\ast} = P_{\mathcal{K}^{\Psi}}$. Moreover, for every $x \in E$, we have
\begin{equation*}
    \Psi(x) = W_{\Psi}^{\ast}\rrho_{\Psi}(x)V_{\psi}|_{\mathcal{D}} = W_{\Psi}^{\ast}U\rrho_{\Phi}(x)V_{\varphi}|_{\mathcal{D}} = W_{\Psi}^{\ast}UW_{\Phi}W_{\Phi}^{\ast}\rrho_{\Phi}(x)V_{\varphi}|_{\mathcal{D}} = W\Phi(x).
\end{equation*}
Hence proved.
\end{proof}
\end{thm}
Let $\Phi\in \mathcal{C}_{loc}(E, C^{*}_{\mathcal{E, F}}(\mathcal{D, O}))$. Then  $\Big( \big(\pi_{\varphi}, V_{\varphi}, \{\mathcal{\mathcal{H}^{\varphi}};\mathcal{\mathcal{E}^{\varphi}}; \mathcal{D}^{\varphi}\}\big), \big(\rrho_{\Phi}, W_{\Phi}, \{\mathcal{K}^{\Phi};\mathcal{F}^{\Phi}; \mathcal{O}^{\Phi}\}\big)\Big)$ is a Stinespring's representation of $(\varphi, \Phi)$. We define the commutant, denoted by  $\rrho_{\Phi}(E)^{\prime}$, of $\rrho_{\Phi}(E)$ as,
\begin{align*}
   \notag \rrho_{\Phi}(E)^{\prime}
  :=\Big\{T\oplus S\in \mathcal{B}(\mathcal{H}^{\varphi}\oplus \mathcal{K}^{\Phi}): \;&    S\rrho_{\Phi}(x)\subseteq \rrho_{\Phi}(x)T \;  \text{and}\\
  &  T\rrho_{\Phi}(x)^*\subseteq \rrho_{\Phi}(x)^*S, \; \text{for all}\; x\in E \Big\}.
\end{align*}
Now, we show some results that are required to establish our main result.
\begin{lemma}\label{Lemma: uniquelydetermined} If $T\oplus S\in \rrho_{\Phi}(E)^{\prime}\cap  C^{*}_{\mathcal{E}^{\varphi}\oplus\mathcal{F}^{\Phi}}(\mathcal{D}^{\varphi}\oplus\mathcal{O}^{\Phi}),$ then for all $x\in E,$ we have the following:
\begin{enumerate}
    \item[(i)] $T^{*}\rrho_{\Phi}(x)^{*}\subseteq \rrho_{\Phi}(x)^*S^*, $
    \item[(ii)]  $S^{*}\rrho_{\Phi}(x)\subseteq \rrho_{\Phi}(x)T^*.$
Moreover $S$ is uniquely determined by $T$.
\end{enumerate}
\end{lemma}
\begin{proof}
In  order to prove (i), let
$\sum\limits_{i=1}^{n}\rrho_{\Phi}(x_{i})V_{\varphi}h_{i}\in \Span\{ \rrho_{\Phi}(E)V_{\varphi}\mathcal{D}\}$ and
$\sum\limits_{j=1}^n \pi_{\varphi}(b_{j})V_{\varphi}g_{j}\in \Span\{\pi_{\varphi}(\mathcal{A})V_{\varphi}\mathcal{D}\}$, for $g_{j} \in \mathcal{A}$, $j \in \{1,2,3,\hdots, m\}$. Since $\mathcal{D}$ is the union space of upward filtered family $\mathcal{E}$, there is $\ell \in \Omega$ with $h_{i}, g_{j} \in \mathcal{H}_{\ell}$ for $i \in \{1,2,3,\hdots,n\}$ and $j \in \{1,2,3,\hdots, m\}$. We see that
\begin{align*}
\Big\langle T^{*}\rrho_{\Phi}(x)^{*}\big(\sum_{i=1}^{n}\rrho_{\Phi}(x_{i})V_{\varphi}h_{i}\big),& \; \sum_{j=1}^{m} \pi_{\varphi}(b_{j})V_{\varphi}g_{j}\Big\rangle_{\mathcal{H}_{\ell}^{\varphi}}\\
&=\Big\langle \rrho_{\Phi}(x)^{*}\big(\sum_{i=1}^{n}\rrho_{\Phi}(x_{i})V_{\varphi}h_{i}\big),\; T\big( \sum_{j=1}^n \pi_{\varphi}(b_{j})V_{\varphi}g_{j}\big) \Big\rangle_{\mathcal{H}_{\ell}^{\varphi}}\\
&=\Big\langle \sum\limits_{i=1}^{n}\rrho_{\Phi}(x_{i})V_{\varphi}h_{i},\; \rrho_{\Phi}(x)T\big( \sum_{j=1}^n \pi_{\varphi}(b_{j})V_{\varphi}g_{j}\big) \Big\rangle_{\mathcal{K}_{\ell}^{\Phi}}\\
&=\Big\langle \sum_{i=1}^{n}\rrho_{\Phi}(x_{i})V_{\varphi}h_{i},\; S\rrho_{\Phi}(x)( \sum_{j=1}^n \pi_{\varphi}(b_{j})V_{\varphi}g_{j}) \Big\rangle_{\mathcal{K}_{\ell}^{\Phi}}\\
&=\Big\langle S^*(\sum_{i=1}^{n}\rrho_{\Phi}(x_{i})V_{\varphi}h_{i}),\; \rrho_{\Phi}(x)( \sum_{j=1}^n \pi_{\varphi}(b_{j})V_{\varphi}g_{j}) \Big\rangle_{\mathcal{K}_{\ell}^{\Phi}}\\
&=\Big\langle \rrho_{\Phi}(x)^*S^*(\sum_{i=1}^{n}\rrho_{\Phi}(x_{i})V_{\varphi}h_{i}),\; \sum_{j=1}^n \pi_{\varphi}(b_{j})V_{\varphi}g_{j} \Big\rangle_{\mathcal{H}_{\ell}^{\varphi}}.
\end{align*}
Since $\Span\{\pi_{\varphi}(\mathcal{A})V_{\varphi}\mathcal{H}_{\ell}\}$ is dense in $\mathcal{H}_{\ell}^{\varphi}$ and $\Span\{\rrho_{\Phi}(E)V_{\varphi}\mathcal{H}_{\ell}\}$ is dense in $\mathcal{K}_{\ell}^{\Phi},$ therefore
$T^{*}\rrho_{\Phi}(x)^{*}|_{\mathcal{K}_{\ell}^{\Phi}}= \rrho_{\Phi}(x)^*S^{*}|_{\mathcal{K}_{\ell}^{\Phi}}$ for every $\ell \in \Omega.$
Hence, we have $T^{*}\rrho_{\Phi}(x)^{*}\subseteq \rrho_{\Phi}(x)^*S^*$. Same reasoning can be applied to prove (ii). The second part follows,
as  $S\rrho _{\Phi }(x)V_{\varphi }h= \rrho _{\Phi }(x)TV_{\varphi }h$ for $x\in E, h\in {\mathcal H}_l$ and
$[\rrho_{\Phi}(E)V_{\varphi}\mathcal{H}_{\ell}]= \mathcal{K}_{\ell}^{\Phi},$ for every $\ell \in \Omega$.

\end{proof}
\begin{prop}\label{Prop: phi-map}
  Let $T\oplus S \in \rrho_{\Phi}(E)^{\prime}\cap C^{*}_{\mathcal{E}^{\varphi}\oplus\mathcal{F}^{\Phi}}(\mathcal{D}^{\varphi}\oplus\mathcal{O}^{\Phi})$ be such that $0\leq T\oplus S\leq I,$ then the map $\Phi_{T\oplus S}: E\to C^{*}_{\mathcal{E, F}}(\mathcal{D, O})$
defined by
 \begin{equation}\label{Eq: Phi T+S}
 \Phi_{T\oplus S}(x)= W_{\Phi}^{*}\sqrt{S}\rrho_{\Phi}(x)\sqrt{T}V_{\varphi}\big|_{\mathcal{D}}, \; \text{for all}\; x \in E,
 \end{equation}
is a $\rvarphi_{T^{2}}$-map. Moreover, $\Phi_{T\oplus S} \preceq \Phi$.
\end{prop}
\begin{proof} We show that the map $\Phi_{T\oplus S}$ defined in Equation (\ref{Eq: Phi T+S}) is well defiend.
Notice that, for each $\ell\in \Omega$, $V_{\varphi}(\mathcal{H}_{\ell})\subseteq \mathcal{H}_{\ell}^{\varphi},$
$\sqrt{T}(\mathcal{H}_{\ell}^{\varphi})\subseteq \mathcal{H}_{\ell}^{\varphi}$ and $\sqrt{S}(\mathcal{K}_{\ell}^{\Phi})\subseteq
\mathcal{K}_{\ell}^{\Phi}.$ Moreover, $W_{\Phi}^*: \mathcal{K}_{\ell}^{\Phi} \to \mathcal{K}_{\ell}$ is an inclusion map.
Thus $\Phi_{T\oplus S}(x)(\mathcal{H}_{\ell})\subseteq \mathcal{K}_{\ell}$, for each $\ell \in \Omega.$ It is easy to see that
$\Phi_{T\oplus S}(x)\big|_{\mathcal{H}_{\ell}}\in \mathcal{B}(\mathcal{H}_{\ell}, \mathcal{K}_{\ell}).$ Let $\eta\in \mathcal{H}_{\ell}^{\perp}\cap \mathcal{D}.$ Then for each $\xi\in \mathcal{K}_{\ell},$ we have
\begin{align*}
\big\langle  W_{\Phi}^{*}\sqrt{S}\rrho_{\Phi}(x)\sqrt{T}V_{\varphi} \eta , \xi\big\rangle_{\mathcal{K}}
&=\big\langle \sqrt{S}\rrho_{\Phi}(x)\sqrt{T}V_{\varphi} \eta , W_{\Phi}\xi\big\rangle_{\mathcal{K}^{\Phi}}\\ \notag
&= \big\langle \rrho_{\Phi}(x)T V_{\varphi} \eta , W_{\Phi}\xi\big\rangle_{\mathcal{K}^{\Phi}}\\ \notag
&=\big\langle \rrho_{\Phi}(x)T\pi_{\varphi}(1)V_{\varphi} \eta , W_{\Phi}\xi\big\rangle_{\mathcal{K}^{\Phi}}\\ \notag
&=\big\langle \rrho_{\Phi}(x)T\pi_{\varphi}(1)V_{\varphi} \eta , W_{\Phi}\xi\big\rangle_{\mathcal{K}^{\Phi}}.
\end{align*}
 We know by Lemma \ref{Lemma: minimalityforperp} that
$\pi_{\varphi}(1)V_{\varphi} \eta\in (\mathcal{H}_{\ell}^{\varphi})^{\perp}\cap \mathcal{D}^{\varphi}$ and $T((\mathcal{H}_{\ell}^{\varphi})^{\perp}\cap \mathcal{D}^{\varphi})\subseteq (\mathcal{H}_{\ell}^{\varphi})^{\perp}\cap \mathcal{D}^{\varphi}$ and $\rrho_{\varphi}(x)((\mathcal{H}_{\ell}^{\varphi})^{\perp}\cap \mathcal{D}^{\varphi})\subseteq
(\mathcal{K}_{\ell}^{\Phi})^{\perp}\cap \mathcal{O}^{\Phi}.$ It implies that $\rrho_{\Phi}(x)T\pi_{\varphi}(1)V_{\varphi} \eta\in (\mathcal{K}_{\ell}^{\Phi})^{\perp}\cap \mathcal{O}^{\Phi}.$

Since $W_{\Phi}(\xi)\in \mathcal{K}_{\ell}^{\Phi},$ we have
\begin{equation*}
\big\langle\Phi_{T\oplus S}(x)(\eta), \xi \big\rangle_{\mathcal{K}}=0\; \text{for all}\; \xi \in \mathcal{K}_{\ell}.
\end{equation*}
It follows that $\Phi_{T\oplus S}(\mathcal{H}_{\ell}^{\perp}\cap \mathcal{D})\subseteq \mathcal{K}_l^{\perp}\cap \mathcal{O}.$
Hence $\Phi_{T\oplus S}: E\to C^{*}_{\mathcal{E, F}}(\mathcal{D, O})$
is a well defined complex linear map.

Now we show that $\Phi_{T\oplus S}$ is a $\rvarphi_{T^2}$-map. Let $x, y\in E.$ Then
\begin{align*}
\Phi_{T\oplus S}(x)^*\Phi_{T\oplus S}(y)&= (W_{\Phi}^{*}\sqrt{S}\rrho_{\Phi}(x)\sqrt{T}V_{\varphi})^* W_{\Phi}^{*}\sqrt{S}\rrho_{\Phi}(y)\sqrt{T}V_{\varphi}\\
&=V_{\varphi}^*\sqrt{T}\rrho_{\Phi}(x)^*\sqrt{S}(W_{\Phi}W_{\Phi}^*)\sqrt{S}\rrho_{\Phi}(y)\sqrt{T}V_{\varphi}\\
&=V_{\varphi}^*\sqrt{T}\rrho_{\Phi}(x)^*\sqrt{S}\sqrt{S}\rrho_{\Phi}(y)\sqrt{T}V_{\varphi}\\
&=V_{\varphi}^*\sqrt{T}\sqrt{T}(\rrho_{\Phi}(x)^*\rrho_{\Phi}(y))\sqrt{T}\sqrt{T}V_{\varphi}\\
&=V_{\varphi}^* T\pi_{\varphi}(\langle x, y\rangle)T V_{\varphi}\\
&=V_{\varphi}^* T^2\pi_{\varphi}(\langle x, y\rangle) V_{\varphi}\\
&=\rvarphi_{T^2}(\langle x, y \rangle).
\end{align*}
Thus $\Phi_{T\oplus S}$ is $\rvarphi_{T^2}$-map and hence $\Phi_{T\oplus S} \in \cloc$. Since $T^2\in \pi_{\mathcal{\varphi}}(\mathcal{A})^{\prime} \cap C^{\ast}_{\mathcal{E}^{\varphi}}
 (\mathcal{D}^{\varphi})$ and $0\leq T^2\leq 1,$ by Proposition \ref{Proposition: phiTisCPandCC}, we conclude that
 $\rvarphi_{T^2} \leq \rvarphi$. Hence by Definition \ref{Definition: orderrelation}, we have $\Phi_{T\oplus S} \preceq \Phi$.
 \end{proof}
The operator $R$ in the following theorem can be considered as the Radon-Nikodym derivative of $\Psi $ with respect to $\Phi .$
\begin{thm}\label{Theorem: RadonNikodymforHilbertmodules}
 Let $\Phi, \Psi\in \mathcal{C}_{loc}(E, C^{\ast}_{\mathcal{E}, \mathcal{F}}(\mathcal{D}, \mathcal{O}))$, where $\Phi $ is a $\varphi $-map, $\Psi $ is a $\psi $-map and $\varphi, \psi \in \cpccloc$. If  $\Psi \preceq \Phi$, then there exists a unique operator
 $R\in \rrho_{\Phi}(E)^{\prime}\cap C^{*}_{\mathcal{E}^{\varphi}\oplus \mathcal{F}^{\Phi}}(\mathcal{D}^{\varphi}\oplus \mathcal{O}^{\Phi})$ with $0\leq R\leq I$ such that $\Psi\sim \Phi_{R}.$
\end{thm}
\begin{proof} Given that $\Psi \preceq \Phi$, by the Definition \ref{Definition: orderrelation}, we have $0\leq \psi\leq\varphi$. In  view of Theorem \ref{Theorem: MainTheroem1}, we construct an operator $T: \mathcal{D}^{\varphi}\to \mathcal{D}^{\psi}$
such that
\begin{equation*}
  T\Big(\sum\limits_{i=1}^{n}\pi_{\varphi}(a_{i})V_{\varphi}h_{i}\Big) = \sum\limits_{i=1}^{n}\pi_{\psi}(a_{i})V_{\psi}h_{i} ~~~
 \;\text{for all}\; h_{i}\in \mathcal{D^{\varphi}}, a_{i}\in \mathcal{A}.
\end{equation*}
Let us define a map $S: \mathcal{O}^{\Phi}\to \mathcal{O}^{\Psi}$ by
$$
S\Big(\sum\limits_{i=1}^{n}\rrho_{\Phi}(x_{i})V_{\varphi}h_{i}\Big)=\sum\limits_{i=1}^{n}\rrho_{\Psi}(x_{i})V_{\psi}h_{i}\;\; \text{for all}\;
h_{i}\in \mathcal{D}^{\varphi}, a_{i}\in \mathcal{A}.
$$
Now for $x_{i}\in E,h_{i}\in \mathcal{H}_{\ell}, i \in \{1,2,3,\cdots n\}$ and for fixed $\ell\in \Omega,$ we have
\begin{align}
  \notag  \Big\| S(\sum\limits_{i=1}^{n}\rrho_{\Phi}(x_{i})V_{\varphi}h_{i}) \Big\|^{2}_{\mathcal{K}_{\ell}^{\Psi}}
  \notag  &=\sum\limits_{i, j=1}^{n}\langle \rrho_{\Psi}(x_{i})V_{\psi}h_{i},\; \rrho_{\Psi}(x_{j})V_{\psi}h_{j}\rangle_{\mathcal{K}_{\ell}^{\Psi}}\\
  \notag  &=\sum\limits_{i, j=1}^{n}\langle h_{i},\; V_{\psi}^*\rrho_{\Psi}(x_{i})^*\rrho_{\Psi}(x_{j})V_{\psi}h_{j}\rangle_{\mathcal{H}_{\ell}}\\
  \notag &= \sum\limits_{i, j=1}^{n}\langle h_{i},\;  V_{\psi}^*\pi_{\psi}(\langle x_{i}, x_{j}\rangle )V_{\psi}h_{j}\rangle_{\mathcal{H}_{\ell}}\\
   \notag  &= \sum\limits_{i, j=1}^{n}\langle h_{i},\;  \psi(\langle x_{i}, x_{j}\rangle) h_{j}\rangle_{\mathcal{H}_{\ell}}\\
    &= \Big\langle
    \begin{bmatrix}h_{1}\\
    \vdots \\
    h_{n}
    \end{bmatrix},\;  \Bigg[\psi\big(\langle x_{i}, x_{j}\rangle\big)\Bigg]_{i,j =1}^{n}  \begin{bmatrix}h_{1}\\
    \vdots \\
    h_{n}
    \end{bmatrix}\Big\rangle_{\mathcal{H}_{\ell}^{n}}.
\end{align}
Since $[\langle x_{i}, x_{j}\rangle]\in M_{n}(\mathcal{A})^{+}$ and $\varphi-\psi\in \mathcal{CP}_{loc}(\mathcal{A},
C^{*}_{\mathcal{E}}(\mathcal{D})),$ thus we have
\begin{align*}
    \sum\limits_{i, j=1}^{n}\Big\langle
    \begin{bmatrix}h_{1}\\
    \vdots \\
    h_{n}
    \end{bmatrix},\;  [\psi(\langle x_{i}, x_{j}\rangle)]  \begin{bmatrix}h_{1}\\
    \vdots \\
    h_{n}
    \end{bmatrix}\Big\rangle_{\mathcal{H}_{\ell}^{n}}
    &\leq  \sum\limits_{i, j=1}^{n}\Big\langle
    \begin{bmatrix}h_{1}\\
    \vdots \\
    h_{n}
    \end{bmatrix},\;  [\varphi(\langle x_{i}, x_{j}\rangle)]  \begin{bmatrix}h_{1}\\
    \vdots \\
    h_{n}
    \end{bmatrix}\Big\rangle_{\mathcal{H}_{\ell}^{n}}\\
    &=\Big\| \sum\limits_{i=1}^{n}\rrho_{\Phi}(x_{i})V_{\varphi}h_{i}\Big\|^{2}_{\mathcal{K}_{\ell}^{\Phi}}.
\end{align*}
 Since $\Span\{\rrho_{\Phi}(E)V_{\varphi}\mathcal{H}_{\ell}\}$
  is dense in $\mathcal{K}_{\ell}^{\Phi},$ thus $S|_{\mathcal{K}_{\ell}^{\Phi}}:{\mathcal{K}_{\ell}^{\Phi}}\to {\mathcal{K}_{\ell}^{\Psi}}$ is a contraction for each $\ell\in \Omega.$ Moreover, since $ \bigcup\limits_{\ell \in \Omega}\mathcal{K}_{\ell}^{\Phi}= \mathcal{O}^{\Phi},$ and $\mathcal{O}^{\Phi}$ is dense
  in $\mathcal{K}^{\Phi},$ therefore $S: \mathcal{K}^{\Phi}\to \mathcal{K}^{\Psi}$ is a contraction such that $S(\mathcal{K}_{\ell}^{\Phi})\subseteq \mathcal{K}_{\ell}^{\Psi}$, for every $\ell\in \Omega$.
  We claim that $S\rrho_{\Phi}(x)\subseteq \rrho_{\Psi}(x) T$ and $\rrho_{\Psi}(x)^*S\subseteq T\rrho_{\Phi}(x)^*.$
  Let $\sum\limits_{i=1}^{n}\pi_{\varphi}(a_{i})V_{\varphi}h_{i}\in \Span\{\pi_{\varphi}(\mathcal{A})V_{\varphi}\mathcal{H}_{\ell}\}.$ Then by direct computation we have
 $$S\rrho_{\Phi}(x)\Big(\sum\limits_{i=1}^{n}\pi_{\varphi}(a_{i})V_{\varphi}h_{i}\Big)=\rrho_{\Psi}(x)T\Big(\sum\limits_{i=1}^{n}\pi_{\varphi}(a_{i})V_{\varphi}h_{i}\Big).$$

  Since $[\pi_{\varphi}(\mathcal{A}) V_{\varphi}\mathcal{H}_{\ell}]= \mathcal{H}_{\ell}^{\varphi},$ for every $\ell\in \Omega$
  and $\mathcal{D}^{\varphi}=\bigcup\limits_{\ell \in \Omega}\mathcal{H}_{\ell}^{\varphi},$ we get that
   $   S\rrho_{\Phi}(x)\subseteq \rrho_{\Psi}(x) T.$
In a similar way, if we take $\sum\limits_{i=1}^{m}\rrho_{\Phi}(y_{i})V_{\varphi}g_{i}\in
\Span\{\rrho_{\Phi}(E)V_{\varphi}\mathcal{D}\}$,  we get
$$\rrho_{\Psi}(x)^{*}S\Big(\sum\limits_{i=1}^n\rrho_{\Phi}(y_{i})V_{\varphi}h_{i}\Big)=T\rrho_{\Phi}(x)^{*}\Big(\sum\limits_{i=1}^n\rrho_{\Phi}(y_{i})V_{\varphi}h_{i}\Big).$$
It follows that $\rrho_{\Psi}(x)^*S|_{\mathcal{K}_{\ell}^{\Phi}}= T\rrho_{\Phi}(x)^*|_{\mathcal{K}_{\ell}^{\Phi}}$ since $\Span\{\rrho_{\Phi}(E)V_{\varphi}\mathcal{H}_{\ell}\}$ is dense in $\mathcal{K}_{\ell}^{\Phi}$, for every $\ell\in \Omega$.
 By using the fact $\mathcal{O}^{\Phi}= \bigcup\limits_{\ell\in \Omega}\mathcal{K}_{\ell}^{\Phi}$, we conclude that
 \begin{equation*}
     \rrho_{\Psi}(x)^{\ast}S\subseteq T\rrho_{\Phi}(x)^{\ast}.
 \end{equation*}
Similarly, $T^{*}\rrho_{\Psi}(x)^{*}\subseteq \rrho_{\Phi}(x)^{*}S^{*}$ and  $S^*\rrho_{\Psi}(x)\subseteq \rrho_{\Phi}(x)T^{*}$.
    Now we show that  $[\rrho_{\Phi}(E)V_{\varphi}\big( (\mathcal{H}_{\ell}^{\perp}\cap \mathcal{D})\big)]=(\mathcal{K}_{\ell}^{\Phi})^{\perp}.$   Let us take  
   $h_{i}\in (\mathcal{H_{\ell}})^{\perp}\cap \mathcal{D}, x_{i}\in E.$
   By Lemma \ref{Lemma: minimalityforperp}, we know that $V_{\varphi}h=\pi_{\varphi}(1)V_{\varphi}h\in (\mathcal{H}_{\ell}^{\varphi})^{\perp}\cap\mathcal{D}^{\varphi}.$ It follows that $\sum\limits_{i=1}^{n}\rrho_{\Phi}(x_{i})V_{\varphi}h_{i} \in(\mathcal{K}_{\ell}^{\Phi})^{\perp}$ so that $[\rrho_{\Phi}(E)V_{\varphi}\big( (\mathcal{H}_{\ell}^{\perp}\cap \mathcal{D})\big)]\subseteq (\mathcal{K}_{\ell}^{\Phi})^{\perp}.$

   Let $\xi_{0}\in (\mathcal{K}_{\ell}^{\Phi})^{\perp}\cap [\rrho_{\Phi}(E)V_{\varphi}\big( (\mathcal{H}_{\ell}^{\perp}\cap \mathcal{D})\big)]^{\bot}.$ Then for all $x\in E, h\in \mathcal{H}_{\ell},$ we know that $\rrho_{\Phi}(x)V_{\varphi}h\in\mathcal{K}_{\ell}^{\Phi}$ and so $\langle \xi_{0}, \;\rrho_{\Phi}(x)V_{\varphi}h\rangle_{\mathcal{K}_{\ell}^{\Phi}}=0$. Moreover, for all $\eta\in \mathcal{H}_{\ell}^{\perp}\cap \mathcal{D}$  we see that
  $\langle \xi_{0}, \;\rrho_{\Phi}(x)V_{\varphi}\eta\rangle_{\mathcal{K}_{\ell}^{\Phi}}=0$. Therefore, we have
  $\langle \xi_{0}, \;\rrho_{\Phi}(x)V_{\varphi}\eta \rangle_{\mathcal{K}^{\Phi}}=0$  for all $x\in E$  and $\eta\in \mathcal{D}.$ Since $\Span\{\rrho_{\Phi}(E)V_{\varphi}(\mathcal{D}\}$ is dense in $\mathcal{K}^{\Phi},$ we have $\xi_{0}=0.$ Hence, we have $[\rrho_{\Phi}(E)V_{\varphi}\big( (\mathcal{H}_{\ell}^{\perp}\cap \mathcal{D})\big)]=(\mathcal{K}_{\ell}^{\Phi})^{\perp}.$

 We show that $S\in C^{*}_{\mathcal{F}^{\Phi}, \mathcal{F}^{\Psi}}(\mathcal{O}^{\Phi}, \mathcal{O}^{\Psi}).$ It is immediate from the
   definition, $S|_{\mathcal{K}_{\ell}^{\Phi}}\in \mathcal{B}(\mathcal{K}_{\ell}^{\Phi},\mathcal{K}_{\ell}^{\Psi})$. Let $x_{i}, y_{j} \in E, h_{i}, \in (\mathcal{H}_{\ell})^{\perp}\cap \mathcal{D}$ and $g_{j} \in \mathcal{H}_{\ell}$ for $i \in \{1, 2, 3, \hdots, n\}, \; j \in \{1,2,3, \hdots, m\}$. Then we have
    \begin{align*}
  \Big \langle S\big(\sum\limits_{i=1}^{n}\rrho_{\Phi}
   (x_{i})V_{\varphi}h_{i}\big), \; \sum\limits_{j=1}^{m}\rrho_{\Psi}(y_{j})V_{\psi}g_{j}\Big\rangle_{\mathcal{K}^{\Psi}}
   &=\sum\limits_{i,j=1}^{n,m}\Big \langle \rrho_{\Psi}(x_{i})V_{\psi}h_{i}, \; \rrho_{\Psi}(y_{j})V_{\psi}g_{j}\Big\rangle_{\mathcal{K}^{\Psi}}\\
   &=\sum\limits_{i,j=1}^{n,m}\Big \langle h_{i}, \; V_{\psi}^*\rrho_{\Psi}(x_{i})^*\rrho_{\Psi}(y_{j})V_{\psi}g_{j}\Big\rangle_{\mathcal{H}}\\
&=\sum\limits_{i,j=1}^{n,m}\Big \langle h_{i}, \; V_{\psi}^*\pi_{\psi}\big(\langle x_{i},y_{j}\rangle \big)V_{\psi}g_{j}\Big\rangle_{\mathcal{H}}\\
&=\sum\limits_{i,j=1}^{n,m}\Big \langle h_{i}, \;\psi\big(\langle x_{i},y_{j}\rangle \big )g_{j}\Big\rangle_{\mathcal{H}}.\\
&=0,
\end{align*}
 since $h_{i}\in \mathcal{H}_{\ell}^{\perp}$ and $\psi(\langle x, y_{j} \rangle)g_{j}\in \mathcal{H}_{\ell}$.
   Now, applying the fact that $\Span\{\rrho_{\Psi}(E)V_{\psi}(\mathcal{H}_{\ell})\}$ is dense $K_{\ell}^{\Psi}$, we get
$S\big(\sum\limits_{i=1}^{n}\rrho_{\Phi}(x_{i})V_{\varphi}h_{i}\big)\in(\mathcal{K}_{\ell}^{\Psi})^{\perp}.$
Since $\Span\{\rrho_{\Phi}(E)V_{\varphi}\big((\mathcal{H}_{\ell})^{\perp}\cap \mathcal{D}\big)\}$ is dense in $(\mathcal{K}_{\ell}^{\Phi})^{\perp}$
and $S$ is bounded, we have $S\big((\mathcal{K}_{\ell}^{\Phi})^{\perp}\big)\subseteq (\mathcal{K}_{\ell}^{\Psi})^{\perp}.$
 Hence  $S\in C^{*}_{\mathcal{F}^{\Phi}, \mathcal{F}^{\Psi}}(\mathcal{O}^{\Phi}, \mathcal{O}^{\Psi}).$

 Let us take $R :=\vert T\vert\oplus \vert S\vert$. Clearly, $R \in C^{\ast}_{\mathcal{E}^{\varphi} \oplus \mathcal{F}^{\Phi}}(\mathcal{D}^{\varphi}\oplus \mathcal{O}^{\varphi})$. We claim that $R\in \rrho_{\Phi}(E)^{\prime}$ and $0\leq R \leq 1.$ Firstly, we observe that
 \begin{equation*}
     S^*S \rrho_{\Phi}(x)\subseteq S^*\rrho_{\Phi}(x)T\subseteq \rrho_{\Phi}(x)T^*T.
 \end{equation*}
 Thus by functional calculus, we see that
  \begin{equation*}
      \vert S\vert \rrho_{\Phi}(x)\subseteq \rrho_{\Phi}(x)\vert T\vert \; \text{and}\; \vert T \vert\rrho_{\Phi}^{*}(x)\subseteq \rrho_{\Phi}^{*}(x)\vert S\vert.
  \end{equation*}
  Therefore $R\in \rrho_{\Phi}(E)^{\prime}$. Since $T$ and $S$ are contractive, thus $0\leq \vert T\vert,\vert S\vert\leq I.$
  Therefore, we have $0\leq R \leq I.$ Next, we show that $\Psi\sim \Phi_{R}$. Recall that the map $\Phi_{R}$ is defined by
  \begin{equation*}
      \Phi_{R}(x)=W_{\Phi}^{*}\vert S\vert^{\frac{1}{2}}\rrho_{\Phi}(x)\vert T\vert^{\frac{1}{2}}V_{\varphi}.
  \end{equation*}
  Moreover, notice that $\vert T\vert\oplus \vert S\vert \in \rrho_{\Phi}(E)^{\prime}\cap C^{*}_{\mathcal{E}^{\varphi}\oplus \mathcal{F}^{\Phi}}(\mathcal{D}^{\varphi}\oplus\mathcal{O}^{\Phi})$ and $0\leq \vert T\vert\oplus \vert S\vert \leq 1.$ Thus by Proposition \ref{Prop: phi-map}, $\Phi_{R}$ is $\rvarphi_{T^{*}T}$-map.
  It is clear from the proof of Theorem \ref{Theorem: MainTheroem1}, $\psi=\rvarphi_{T^*T}.$ We conclude that
  \begin{equation}\label{Equation: R}
      \Psi\sim \Phi_{R}.
  \end{equation}
To see uniqueness, suppose that $R'= T'\oplus S'\in \rrho_{\Phi}(E)^{\prime}\cap C^{*}_{\mathcal{E}^{\varphi}\oplus \mathcal{F}^{\Phi}}(\mathcal{D}^{\varphi}\oplus \mathcal{O}^{\Phi})$, $0 \leq R^{\prime} \leq I$ such that
\begin{equation}\label{Equation: R^prime}
    \Psi \sim \Phi _{R'}.
\end{equation}
Then by Equations (\ref{Equation: R}), (\ref{Equation: R^prime}), we get $\Phi_{R} \sim \Phi_{R^{\prime}}$, whence we deduce from Definition \ref{Definition: equivalent} that $\rvarphi_{T^{\ast}T} = \rvarphi_{{T^{\prime}}^{2}}$. Thus $T^{\ast}T = {T^{\prime}}^{2}$ by Corollary \ref{isomorphism}. 
 We conclude that $R = R^{\prime}$ from the second part of Lemma \ref{Lemma: uniquelydetermined}.
\end{proof}
It is worth mentioning that the article \cite{Moslehian et.al} deals with Stinespring's and Radon-Nikodym type representations for completely positive $n\times n$ matrix of linear maps on Hilbert locally $C^{\ast}$-modules by considering the existence of positive elements in locally $C^{\ast}$-algebras. Whereas, Theorem \ref{Theorem: Stinespringforphimaps}, \ref{Theorem: RadonNikodymforHilbertmodules} are different from their setup as we consider the existence of  local positive elements in locally $C^{\ast}$-algebras.

\section*{Acknowledgements}
The first author thanks SERB(Science and Engineering Research Board, India) for funding through J C Bose Fellowship project. The third named author thanks NBHM (National Board for Higher Mathematics, India) for financial support with ref No.0204/66/2017/R\&D-II/15350. Also, both second and third named authors would like to thank the Indian Statistical Institute Bangalore for providing necessary facilities to carry out this work.

We express our sincere thanks to the anonymous reviewer for valuable comments.


\end{document}